\documentclass{article}

\usepackage{color}
\usepackage{bm}
\usepackage{graphicx}
\usepackage{subfigure}
\usepackage{amssymb,amsmath,amsthm} 
\usepackage{natbib}
\usepackage{float}
\usepackage{enumerate}
\usepackage{currfile}

\newtheorem{thm}{Theorem}
\newtheorem{remark}{Remark}
\newtheorem{prop}{Proposition}
\newtheorem{lem}{Lemma}

\newtheorem{condition}{Condition}
\newtheorem{example}{Example}

\newcommand{\cov}{\mathrm{cov}}

\def\T{{\ensuremath{\mathsf{T}}}}
\newcommand{\I}{\mathcal{I}}
\newcommand{\F}{\mathcal{F}}

\newcommand{\p}{\partial}
\newcommand{\X}{\mathbf{X}}
\newcommand{\x}{\mathbf{x}}

\newcommand{\E}{\mathbb{E}}
\newcommand{\R}{\mathbb{R}}
\newcommand{\Rp}{\mathbb{R}_{+}}

\newcommand{\y}{\mathbf{y}}
\newcommand{\calf}{\mathcal{F}}
\begin{document}

\title{Asymptotic Distribution of the Score Test for Detecting Marks in Hawkes Processes}

\author{Simon Clinet\footnote{Faculty of Economics, Keio University. 2-15-45 Mita, Minato-ku, Tokyo, 108-8345, Japan. Phone:  +81-3-5427-1506. E-mail: clinet@keio.jp, website: http://user.keio.ac.jp/\char`\~clinet/}, William T.M. Dunsmuir\footnote{Corresponding Author, School of Mathematics and Statistics, University of New South Wales, Sydney, NSW 2052, Australia. E-mail: W.Dunsmuir@unsw.edu.au, website: https://research.unsw.edu.au/people/professor-william-t-m-dunsmuir.},\\ Gareth W. Peters\footnote{Department of Actuarial Mathematics and Statistics, Heriot-Watt University, Edinburgh, UK. E-mail: g.peters@hw.ac.uk website: https://www.hw.ac.uk/staff/uk/macs/gareth-peters.htm.}, Kylie-Anne Richards\footnote{UTS Business School, University of Technology Sydney, Sydney, Australia. E-mail: Kylie-Anne.Richards@uts.edu.au,  website: https://www.uts.edu.au/staff/kylie-anne.richards-0. } }
\maketitle

\begin{abstract} 
The asymptotic distribution of the score test of the null hypothesis that marks do not impact the intensity of a Hawkes marked self-exciting point process is shown to be chi-squared. For local asymptotic power, the distribution against local alternatives is also established as non-central chi-squared. These asymptotic results are derived using existing asymptotic results for likelihood estimates of the unmarked Hawkes process model together with mild additional conditions on the moments and ergodicity of the marks process and an additional uniform boundedness assumption, shown to be true for the exponential decay Hawkes process.
\end{abstract}
\noindent%
\newpage
\section{Introduction}\label{Sec: Introduction}
Since their introduction over fifty years ago Hawkes self exciting process models \citep{Hawkes1971} have been used to model point processes in many fields of application including seismology \citep{Ogata1988_PowerLawKernel}, sociology \citep{Crane2008_HawkesSocial}, modelling neuronal systems and increasingly in recent years for modelling high frequency financial trading 
(for a general review, see \cite{bacry2015hawkes,Hawkes2018_FinanceReview}). Extensions of the Hawkes process where parameters are time-varying and replicate the non-stationarity of intraday financial data have also been considered in \cite{chen2013inference, clinet2018statistical} for example. The theoretical properties of such models are quite well advanced as is estimation methodology and its associated statistical theory. Increasingly \textit{marked} Hawkes processes, in which marks attached to past event times influence future intensities, are being considered for a range of applications. For example, \cite{RichDunsmuirPeters} consider the use of marked Hawkes processes for modelling millisecond recordings of activity in the limit order book for a range of assets traded on international futures markets. In these applications there are numerous potential marks that are recorded at each event and a method is required to efficiently screen out those that are not influential on future event arrival intensities before the joint models for the event times and associated marks are estimated. 

Assessment of influential marks could be done by simultaneously estimating the parameters of the marked Hawkes process and then assessing them for statistical significance. Even if there is a single scalar valued mark included in the model estimation using, say, maximum likelihood methods, there are computational challenges. When numerous marks are jointly included in the model these challenges are substantial. The use of likelihood for the marked Hawkes process model, if feasible, would allow use of standard inferential techniques such as the Wald or likelihood ratio test for testing the significance of marks impact.  However, the relevant statistical inference methods and theory for marked Hawkes processes are not well developed at this time. 

An alternative approach for assessing the impact of marks on intensity is the score test as proposed in \cite{RichDunsmuirPeters} where details of computational implementation, simulation and application to limit order book event series is presented. As is well known, this test, also known as the Lagrange Multiplier test \citep{BreuschPaganLMtest}, is computed using the score of the likelihood evaluated under the null hypothesis, which, in this application, is that marks do not impact the intensity, so that the event times are that of an unmarked Hawkes process. The score statistic can then be constructed easily based on a single fitted intensity for an unmarked process. Because of this, the score test leads to substantial computational advantages particularly when relevant and significant marks need to be selected from a possibly large catalogue before requiring the effort of jointly fitting the marked process model.  Apart from the obvious computational advantage afforded by using a single Hawkes model fit, the score test large sample distribution theory can be derived using existing asymptotic theory for \textit{unmarked} Hawkes processes as is explained below. In this paper the large sample distribution of the score test under the null hypothesis that the mark or marks under test do not boost the intensity of events is shown to be the standard chi-squared distribution with appropriate degrees of freedom. We also show that the power against local alternatives is non-central chi-squared.

In the literature on the theory and inferential methods for marked Hawkes processes focus has been on the situation where marks are unpredictable (in a sense to be detailed below), such as when they are independent and identically distributed. Our extensive experience with applications in high frequency financial data suggests that influential marks display serial dependence in addition to cross dependence between the marks. Accordingly, we derive the asymptotic distribution of the score test for serially dependent multivariate valued marks. This requires us to show that stable marked Hawkes process exist when the marks are stationary serially dependent, something that is not currently available in the literature. 

From now on we consider a univariate Hawkes self exciting marked point process (SEPP) $N_{g} \in \mathbb{N}\times\mathbb{X}$, observed over the interval $t \in [0,T]$ and which takes the value $0$ at $t=0$. There are $N_T$ events observed in the interval $[0,T]$ at times $0<t_1 <t_2< \ldots <t_{N_T} \le T$ and a vector of $d$ marks $\X_i \in \mathbb{X}\subset\mathbb{R}^d$ is associated with the $i$th event. The observed points of this process are $\{(t_i,\x_i), i = 1,\dots,N_T\}$. In \cite{RichDunsmuirPeters} relevant marks constitute a vector of correlated marks which are also serially dependent. In order to accommodate such examples we explain how to define a marked Hawkes process with serially dependent marks, give conditions for stationarity of the point process, and, define the relevant quasi-likelihood. 

Following \cite{liniger2009multivariate} and \cite{embrechts2011}, with modifications to notation as used in \cite{ClinetYoshida2017}, let the marked Hawkes SEPP have intensity process given by 
\begin{equation} \label{Eqn: Intensity Process}
\lambda_g(t;\theta,\phi,\psi) = \eta + \vartheta \int_{[0,t)\times\mathbb{X}} w(t-s;\alpha) g(\x;\phi,\psi)N_g(ds \times d\x)
\end{equation}
where $w$ is a non-negative decay function satisfying
\begin{equation}\label{Eqn: integral omega, s omega} 
\int_{0}^{\infty}w(s;\alpha)ds=1, \quad \int_{0}^{\infty}sw(s;\alpha)ds < \infty.
\end{equation}
The immigration rate is $\eta$, the branching coefficient $\vartheta$ and the parameter $\alpha$, not necessarily scalar, specifies the decay function $w$. The marks $\X$ have density $f(\x;\phi)$ (w.r.t. Lebesgue measure) and impact the intensity through the scalar valued boost function $g(\X;\phi,\psi)$, where $g(\cdot;\phi,\psi): \R^d\to \R^{+}$ and $\psi$ is a vector parameter of length $r$ specifying the way in which marks enter the boost function. In addition, $g$ depends on the parameters of the marks density, $\phi$, because the normalization $\E_{\phi}[g(\X;\phi,\psi)]=1$ is required to obtain a stationary solution to \eqref{Eqn: Intensity Process} as in \cite{embrechts2011}. In Section \ref{SubSec: Existence of Stable Marked Hawkes} we show, for the case where the marks are i.i.d as considered in \cite{embrechts2011}, that under this normalization a stationary solution exists and, for serially dependent marks under a stronger condition on the conditional expectation of $g(\x;\phi,\psi)$ given the past, this is also true.

Henceforth, let $\theta=(\eta, \vartheta,\alpha)\in \Theta$, $\phi\in \Phi$ and $\psi \in \Psi$ for some parameter spaces $\Theta$, $\Phi$ and $\Psi$. Let $\nu = (\theta,\phi,\psi) \in \Theta \times \Phi \times \Psi$ be the collection of all parameters for the marked process with intensity function \eqref{Eqn: Intensity Process}. We denote the true value of the parameters as $(\theta^{*}, \phi^{*}, \psi^*)$.
When the null hypothesis holds, the true parameter vector is denoted $\nu^{*}=(\theta^{*},\phi^{*},0)$ and it is assumed that $g(\x;\phi,0)\equiv 1$ so that the intensity in \eqref{Eqn: Intensity Process} does not depend on the marks. Since the score test will be developed under the null hypothesis we only require details of $\Theta$ in the derivation and theory to follow. Specifically we let $\Theta$ be a finite dimensional relatively compact open subset of $\R^K$, $K>1$. The parameter space $\Phi$ for the marks density will typically be the natural space of parameters for the specified density and the boost parameter space $\Psi$ is chosen as appropriate for the form of $g$.

Quite general normalized boost functions, $g(\X;\phi,\psi)$ can be constructed by starting with a function $h(\X,\psi)$ and defining the boost function
\begin{equation}\label{Eqn: Boost function g}
g(\X;\phi,\psi)= \frac{h(\X;\psi)}{\mathbb{E}_{\phi}[h(\X;\psi)]}.
\end{equation}
It is no loss of generality to require $h(\X; 0) \equiv 1$.
Many examples of boost functions including the polynomial, exponential and power functions forms for $h$, as presented in \cite{liniger2009multivariate}, as well as additive and multiplicative combinations of individual elements of $\X$ used in \cite{RichDunsmuirPeters} can be formulated in this way. The null hypothesis being assessed with the score test is $H_0: \psi=0$, which is equivalent to $g(\x;\phi,0) = 1$ so that marks do not boost intensity.

Under $H_0$  the observed event times are those of an unmarked Hawkes SEPP, $N$ with intensity denoted by 
\begin{equation} \label{Eqn: Intensity Process Null Hyp}
\lambda(t;\theta) = \eta + \vartheta \int_{[0,t)} w(t-s;\alpha)N(ds).
\end{equation}
We assume that the initial value of the intensity is $\lambda(t_0)= C_0$ for some specified value of $C_0$;  for example $C_0=\mathbb{E}[\lambda(t)]=\eta/(1-\vartheta)$, the theoretical long run average for a stationary Hawkes process \citep{laub2015hawkes}. Note also that this intensity function is not defined using events prior to the observation period, that is for $t<0$, because in practice \eqref{Eqn: Intensity Process} is used for computation of the likelihood. The intensity process defined in \cite{embrechts2011} is the stationary version with infinite, but unobserved, event history included. In \cite{bremaud1996} the authors show that a suitable probability space exists on which a stationary version on $\R$, $N^{\infty}$, can be defined and to which the non-stationary version in \eqref{Eqn: Intensity Process} converges. We assume that $N^{\infty}$ is ergodic and this is proven in \cite{ClinetYoshida2017} for the exponential decay case. \cite{ogata1978asymptotic} considers both the stationary version of the intensity process and the non-stationary version as in \eqref{Eqn: Intensity Process Null Hyp} along with the associated likelihoods.

The remainder is organized as follows. Section \ref{SubSec: Existence of Stable Marked Hawkes} proves (see Proposition \ref{propExistence}), via a thinning construction, that a stationary marked Hawkes process can be constructed when the marks are observed from a continuous time stationary process and gives several examples of processes for which the conditions are met. Section \ref{Sec: Quasi-likelihood} extends the definition of the joint likelihood of event times and marks beyond the i.i.d. case currently available in the literature to marks which are serially dependent. Section \ref{Sec: The Score Test} defines the score test in detail. Section \ref{Sec: Asymptotic Dist Score Test} states the main result (see Theorem \ref{Thm: Score Statistic Asymptotic Chi-squared}) that the score statistic is asymptotically chi-squared distributed under the null hypothesis that marks do not impact the intensity function. For this result, in addition to the conditions of \cite{ClinetYoshida2017} for the consistency and asymptotic normality of the unmarked process, conditions are required on the existence of moments and ergodicity of the mark process itself together with an additional condition (Condition \ref{Cond: Condition on sup lambda_g}) which links the marks and the unmarked intensity process. Lemma \ref{Lem 1} shows that the Condition \ref{Cond: Condition on sup lambda_g} is satisfied for the case of exponential decay function $w$. Section \ref{Sec: Local Power} proves that the score statistic is asymptotically non-central chi-squared distributed under local alternatives of the form $\psi_T^*=\gamma^*/\sqrt{T}$ where $T \to \infty$ and $T$ is the length of the interval over which the point process is observed. Section \ref{Sec: Conclusions etc} discusses possible extensions to the main results. The Appendices contain proofs.

\section{Existence of a stable marked Hawkes process}\label{SubSec: Existence of Stable Marked Hawkes}

Assume the existence of a probability space $(\Omega, \calf, \mathbb{P})$ bearing a continuous time process $(\y_t)_{t \in \mathbb{R}}$ taking values in the Borel space $(\mathbb{X}, \mathcal{X})$ (meaning that there exists a bijection $h$ between $\mathbb{X}$ and $[0,1]$ such that $h$ and $h^{-1}$ are measurable, see \citep[p.7]{kallenberg2006foundations}). We define $(\calf_t^\y)_{t \in \mathbb{R}}$ the canonical filtration of $\y$. We now give a proof of the existence of the marked Hawkes process using a thinning method similar to \citep[Chapter 6]{liniger2009multivariate} and \cite{bremaud1996}. To that end, we assume the existence on $(\Omega,\calf,\mathbb{P})$ of a Poisson process $\overline{N}$ with intensity $1$ on $\mathbb{R}^2$, with points denoted $(t_i,u_i)_{i \in \mathbb{Z}}$ and independent of $\y$. Then we consider the canonical process $(t_i, u_i, \y_{t_i})_{i \in \mathbb{Z}}$ and see this process as a random measure $\overline{N}_g$ on $\mathbb{R}^2 \times \mathbb{X}$. We associate to $\overline{N}$ the filtration generated by the $\sigma$-algebras $\calf_t^{\overline{N}}=\sigma \{ \overline{N}((-\infty,s] \times A), s \in (-\infty,t], A \in \mathcal{B}(\mathbb{R})\}$, where $\mathcal{B}(\mathbb{R})$ is the Borel $\sigma$-field of $\mathbb{R}$. Similarly, we associate to $\overline{N}_g$ the filtration generated by the $\sigma$-algebras $\calf_t^{\overline{N}_g}=\sigma \{ \overline{N}_g((-\infty,s] \times A \times B), s \in (-\infty,t], A \in \mathcal{B}(\mathbb{R}), B \in \mathcal{X}\}$. Note that $\overline{N}(dt \times du) = \overline{N}_g(dt \times du \times \mathbb{X})$.  Consider the filtration generated by $\calf_t = \calf_t^{\y} \vee \calf_t^{\overline{N}}$. In Proposition \ref{propExistence} below, we show the existence of a marked point process $N_g$ adapted to $\calf_t$, satisfying (\ref{Eqn: Intensity Process}), and which is constructed as an integral over the canonical measure $\overline{N}_g$. 

Before we state our result, we recall that a marked point process of the form $(\tau_i, \y_{\tau_i})_{i \in \mathbb{Z}}$ is stationary if for any $t \in \mathbb{R}$, it has the same distribution (seen as a random measure) as $(\tau_i + t , \y_{\tau_i + t})_{i \in \mathbb{Z}}$. 

\begin{prop} \label{propExistence} Assume that for any $t \in \mathbb{R}$, $\E[g(\y_t; \phi,\psi) | \calf_{t-}^{\y}] \le C$ where $C < 1/\vartheta $, along with $\int_0^{+\infty} w(s;\alpha)ds = 1$.  Then, there exists a marked point process of the form $(\tau_i, \x_i)_{i \in \mathbb{Z}} := (\tau_i, \y_{\tau_i})_{i \in \mathbb{Z}}$, also represented by the random measure $N_g$ on $\mathbb{R} \times \mathbb{X}$, such that:
	\begin{enumerate}
		\item The counting process associated to $(\tau_i)_{i \in \mathbb{Z}}$ is adapted to $(\calf_t)_{t \in \mathbb{R}}$, and admits the stochastic intensity (with respect to $\calf_t$)
		$$ \lambda_g(t) = \eta + \vartheta \int_{(-\infty,t) \times \mathbb{X} } w(t-s;\alpha) g(\x; \phi,\psi)N_g(ds \times d\x).$$
		\item The random measure $N_g$ admits $\pi(ds \times d\x) = \lambda_g(s)ds \times F_s(d\x)$ as predictable compensator, where $F_s(d\x)$ is the conditional distribution of $\y_s$ given $\calf_{s-}^{\y}$. We recall that by definition, for any non-negative measurable predictable process $W$ (see \citep[Theorem II.1.8]{jacod2013limit}), the process\\ $\int_{(-\infty,t] \times \mathbb{X}} W(s,\x) \pi(ds \times d\x)$ is predictable, and moreover for any $u \leq t$
			$$ \E \left[\int_{(u,t] \times \mathbb{X}} W(s,\x) N_g(ds \times d\x) | \calf_{u}\right] = \E  \left[\int_{(u,t] \times \mathbb{X}} W(s,\x) \pi(ds \times d\x) | \calf_{u}\right].$$
		\item If the process $(\y_t)_{t \in \mathbb{R}}$ is stationary, then so is $N_g$.
	\end{enumerate} 
\end{prop}
\begin{proof}[Proof of Proposition \ref{propExistence}]
	We construct by thinning and a fixed point argument the marked point process $N_g$. Define $\lambda_{g,0}(t) = \eta$. By induction, we then define the sequences of processes $(N_{g,n})_{n \in \mathbb{N}}$ and $(\lambda_{g,n})_{n \in \mathbb{N}}$ as follows. For any $t \in \mathbb{R}$ and $A \in \mathcal{X}$
	\begin{subequations}
		\begin{align}
		N_{g,n}((-\infty, t] \times A) &= \int_{(\-\infty,t] \times \mathbb{R} \times A}{\mathbf{1}_{\{0\leq u \leq \lambda_{g,n}(s)\}} \overline{N}_g(ds \times du \times d\x)},\label{Eqn: 5a}\\
		\label{lambdan1} \lambda_{g,n+1}(t) &= \eta + \vartheta \int_{(-\infty,t) \times \mathbb{X} } w(t-s;\alpha) g(\x; \phi,\psi)N_{g,n}(ds \times d\x).
		\end{align} 
		\label{systemHawkes}
	\end{subequations}
By taking conditional expectations throughout \eqref{Eqn: 5a} it is immediate to see that $t \to \lambda_{g,n}(t)$ is the stochastic intensity of $t \to N_{g,n}((-\infty,t] \times \mathbb{X})$. Moreover, by positivity of $w(t-s;\alpha) g(\x; \phi,\psi)$, we immediately deduce that the point process $N_{g,n}((-\infty, t] \times A)$ and the stochastic intensity $\lambda_{g,n}(t)$ are point wise increasing with $n$, so that we may define $N_g$ and $\lambda_{g}$ their limit processes. Moreover, by the monotone convergence theorem, taking the limit $n \to +\infty$ in the above equations yields that $t \to \lambda_g(t)$ is the stochastic intensity of $t \to N_g((-\infty, t] \times \mathbb{X})$ and has the desired shape. All we have to check to get the first claim of the proposition is the finiteness of the two limit processes. Let $\rho_n(t) = \E[\lambda_{g,n}(t) - \lambda_{g,n-1}(t)]$. We have
	\begin{eqnarray*}
		\rho_n(t) &=& \vartheta \E\left[ \int_{(-\infty,t) \times \mathbb{X} } w(t-s;\alpha) g(\x; \phi,\psi)\{N_{g,n-1} - N_{g,n-2}\}(ds \times d\x) \right] \\
		&=&  \vartheta \E\left[ \int_{(-\infty,t) \times \mathbb{X} } w(t-s;\alpha) g(\x; \phi,\psi)\{\lambda_{g,n-1}(s) - \lambda_{g,n-2}(s)\}ds F_s(d\x) \right] \\
		&=& \vartheta \E \left[ \int_{(-\infty,t)  } w(t-s;\alpha) \E[g(\y_{s}; \phi,\psi) | \calf_{s-}^\y] \{\lambda_{g,n-1}(s) - \lambda_{g,n-2}(s)\} ds \right] \\
		& \leq & C\vartheta   \int_{(-\infty,t)  } w(t-s;\alpha) \rho_{n-1}(s) ds,\\
	\end{eqnarray*}
	where we have used the $\calf_{s-}$ measurability of the stochastic intensities and that $\E[g(\y_{s}; \phi,\psi) | \calf_{s-}] = \E[g(\y_{s}; \phi,\psi) | \calf_{s-}^\y] \leq C$ (by independence of $\y$ and $\overline{N}$). From here, we deduce that $\sup_{s \in (-\infty,t)} \rho_n(s) \leq C\vartheta  \sup_{s \in (-\infty,t)} \rho_{n-1}(s)$ since $\int_0^{+\infty} w(s;\alpha)ds < 1 $. By a similar calculation, we also have $\rho_1(t) \leq C \vartheta \eta$, so that  by an immediate induction $\sup_{s \in (-\infty,t)} \rho_n(s) \leq (C\vartheta)^n \eta$. Therefore, $\E \lambda_g(t) = \eta + \sum_{k=1}^{+\infty} \rho_k(t) \leq \eta/(1-C\vartheta)  < +\infty$, which implies the almost sure finiteness of both $N_g$ (on any set of the form $[t_1,t_2] \times \mathbb{X}$, $- \infty < t_1 \leq t_2 < + \infty$) and $\lambda_g(t)$. This proves the first claim. Now we prove the second point. By a monotone class argument, it is sufficient to take $W(s,\x) = \mathbf{1}_A(\x)$ and prove the martingale property for any $A \in \mathcal{X}$. We show that for any $n \in \mathbb{N}$, $\pi^n(ds \times d\x) = \lambda_{g,n}(s)ds \times F_s(d\x)$ is the compensator of $N_{g,n}$. For $n=0$, we have 
		\begin{eqnarray*}
			\E \left[\int_{(u,t] \times \mathbb{X}} W(s,\x) N_{g,0}(ds \times d\x) | \calf_{u}\right] &=& \E \left[\sum_{u < t_i \leq t} \mathbf{1}_A(\y_{t_i})| \calf_{u}\right]\\
			&=&\E \left[\sum_{u < t_i \leq t} \E[ \mathbf{1}_A(\y_{t_i})| \calf_{t_i-} \vee \calf_{+\infty}^{\overline{N}}] | \calf_{u}\right]\\
			&=&\E \left[\sum_{u < t_i \leq t} \E[ \mathbf{1}_A(\y_{t_i})| \calf_{t_i-}^{\y} ] | \calf_{u}\right]\\
			&=&\E \left[\int_{(u,t]}\E[ \mathbf{1}_A(\y_{s})| \calf_{s-}^{\y} ] N_{g,0}(ds\times\mathbb{X}) | \calf_{u}\right] \\
			&=& \E \left[\int_{(u,t] \times \mathbb{X}} \mathbf{1}_A(\x) \lambda_{g,0}(s) F_s(d\x)ds | \calf_{u}\right]\\
			&=&\E \left[\int_{(u,t]}W(s,\x) \pi^0(ds \times d\x) | \calf_{u}\right],\\
		\end{eqnarray*} 
		where the existence of a regular distribution for $\y_s$ given $\calf_{s-}^\y$ is a consequence of the fact that $\mathbb{X}$ is a Borel space along with \citep[Theorem 6.3]{kallenberg2006foundations}. By induction, we easily prove that this holds for any $n \in \mathbb{N}$, and thus the second claim is a direct consequence of the monotone convergence theorem. Finally, the third claim comes from the stationarity of $\lambda_{g,n}$ and $N_{g,n}$ which is in turn a consequence of the stationarity of $\overline{N}_g$ and $\y$. 
\end{proof}

\begin{remark}
The above construction also works for a marked point process starting from $0$ instead of $-\infty$ (just replace $-\infty$ by $0$ in all the integrals). In that case, the resulting process $N_g$ is obviously not stationary, but one can prove that $N_g$ converges to the stationary version starting from $-\infty$ by a straightforward adaptation of the proof of Theorem 1 in \cite{bremaud1996}.
\end{remark}

\begin{remark}
If the marks $\y_s$ are i.i.d, the conditional expectation $\E[g(\y_s) | \calf_{s-}^\y]$ reduces to the usual expectation $\E[g(\y_s)]$ which is equal to unity due to normalization of the boost function $g$. Hence the condition that $\E[g(\y_s)|\calf_{s-}^{\y}] \leq C <1/\vartheta$ is obviously satisfied.
\end{remark}

\begin{remark}
 When $\y$ is a left-continuous process, or more generally a predictable process, then $\E[g(\y_s)|\calf_{s-}^{\y}]=g(\y_s)$ which leads to the condition $g(\y_s)\leq C <1/\vartheta$. This may be very restrictive in practice. For example, for a parametric linear boost in a single mark this would require the mark process to be bounded above by a constant which depends on $\phi$, $\psi$ and $\vartheta$. It is also interesting to note that for mark processes with continuous sample paths, mixing conditions (which specify the rate at which dependence fades away with increasing time separation) will not lead to a weakening of the aforementioned stringent condition. The difficulty stems from the dependence of $g(\y_t)$ and $g(\y_s)$ when $t$ and $s$ are close together. If at some time $t_0$, $g(\y_{t_0})>1/\vartheta$, by `continuity', it will stay above that level for some time $[t_0, t_0+\epsilon]$. On this interval the process becomes explosive, and regardless of the number of jumps, all the marks are highly correlated since $g(\y_s)\approx g(\y_{t_0}) > 1/\vartheta$ for $s \in [t_0, t_0+\epsilon]$.
\end{remark}

\begin{remark}
In view of the last remark, marks which arise from a stochastic process with continuous sample paths are probably not practical for use in marked Hawkes self exciting processes. On the other hand, marks based on a stochastic process which contains some degree of independence could more easily satisfy the condition of Proposition \ref{propExistence}. For example, let $\y_t = U_t + V_t$ where $U_t$ has continuous sample paths and $V_t$ be a pure noise process independent of $U_t$. More generally a conditionally independent specification in which $U_t$ is as before and $\y_t | \{U_t\} \overset{i.i.d}{\sim} f(\cdot|U_t)$ would also more easily satisfy the condition. For instance $U_t$ may specify some of the parameters needed for the density $f$. 
\end{remark}

\begin{remark}
The condition $\E[g(\y_t)|\calf_{t-}^{\y}] \leq C <1/\vartheta$ can be slightly relaxed as we next explain. Let $\rho_n(t) = \E[\lambda_{g,n}(t) - \lambda_{g,n-1}(t)]$. Following the proof of Proposition 1 we know that a sufficient condition for non-explosion is $\sum_{n=1}^{+\infty} \rho_n(t)<\infty$. A straightforward induction shows that
\[
\rho_n(t) = \vartheta^n \eta \E\int_{-\infty}^{t_1}\cdots \int_{-\infty}^{t_{n-1}} w(t-t_1) \cdots w(t_{n-1}-t_{N_T})\psi_{t_1} \cdots \psi_{t_{N_T}}dt_1 \cdots dt_{N_T},
\]
where $\psi_t = \E[g(\y_t)|\calf_{t-}^{\y}]$. Therefore, if we replace the above condition by $\int_{-\infty}^t w(t-s)\E[g(\y_s)|\calf_{s-}^{\y}]ds \leq C < 1/\vartheta$ for any $t \in \mathbb{R}$, then $\sum_{n=1}^{+\infty} \rho_n(t) \leq \frac{\eta}{1-C\vartheta} < +\infty$, and the process is stable.
\end{remark}

\begin{remark}
If the marks are considered to be observations on a discrete time process then it is not obvious that the thinning method used above can be used to construct a marked Hawkes process. Moreover, while it is possible to construct a non-stationary marked Hawkes process with discrete marks by iterating the intensity function from some initial time $t_0$, it is not clear whether there exists a stationary version of this process on $\mathbb{R}$. 
\end{remark}

\section{Quasi-likelihood for marked Hawkes processes}\label{Sec: Quasi-likelihood}

The log-likelihood and associated statistical properties for the unmarked Hawkes SEPP has a long history -- see \cite{ozaki1979maximum}, \cite{ogata1978asymptotic} or \cite{andersen1996statistical} for example. In deriving the likelihood \citep[Definition 3]{embrechts2011} assume that the marks are unpredictable as defined in \citep[Definition 6.4.III(b)]{daley2002introduction} so that the distribution of $\X_i$, the mark at time $t_i$, is independent of previous event times and marks, i.e. of $\{(t_j, \X_j)\}$ for $t_j <t_i$. An example of unpredictable marks is where the marks are conditionally i.i.d. given the past of the process but the marks may impact on the future of the intensity $\lambda_g$ as in \eqref{Eqn: Intensity Process}. The simplest example of this is where the marks are actually i.i.d. unconditionally as considered in \cite{embrechts2011}.
In our empirical analysis we have frequently observed that $\{\X_i\}$ is a time series of serially dependent marks. In this case the unpredictability property does not hold. As far as we can determine, in the literature on likelihood inference for marked Hawkes processes there is no existing treatment of the serially dependent marks case.

In general, it is possible to represent the log-likelihood $\bar{l}_g$ when the marks are not i.i.d as follows: recall that the integer-valued measure $N_g(dt \times d\x)$ admits a predictable compensator $\pi(dt \times d\x)$ by Proposition \ref{propExistence} (ii) of the form $\pi(ds \times d\x) = \lambda_g(s;\nu)ds \times F_s(d\x,\phi)$. Assuming that for any $s \in \mathbb{R}_+$, the conditional distributions $F_s(d\x,\phi)$ are dominated by some measure $c(d\x)$ ($F_s(d\x,\phi) = f_s(\x;\phi)c(d\x)$), using \citep[Theorem III.5.19]{jacod2013limit}) we can generalize  the log-likelihood  \eqref{Eqn: log likelihood Unmarked Hawkes} for the pure point process with 
$$ \bar{l}_g(\nu) = \int_{[0,T] \times\mathbb{X}} \ln [\lambda_g(t;\nu)f_t(\x;\phi)]N_g(dt\times d\x) - \int_{[0,T]}\underbrace{\int_{\mathbb{X}}f_t(\x;\phi)c(d\x)}_{=1}\lambda_g(t;\nu)dt.$$
Expanding the logarithm, we get
\begin{equation}\label{Eqn: log likelihood General Marked Case}
 \bar{l}_g(\nu) = \int_{[0,T] \times\mathbb{X}} \ln \lambda_g(t;\nu)N_g(dt\times d\x) - \Lambda_g(T;\nu) + \int_{[0,T] \times\mathbb{X}} \ln f_t(\x;\phi)N_g(dt\times d\x),
\end{equation}
where the compensator at $T$ is 
\begin{equation}
\Lambda_g(T;\nu)=\int_{[0,T]}\lambda_g(t;\nu)dt.\nonumber
\end{equation}
However, because of the third term, computing (\ref{Eqn: log likelihood General Marked Case}) requires that one observes the whole trajectory of the joint process $(N_t, \y_t)_{t \in [0,T]}$. When assuming that we only have discrete observations of the form $(t_i, \x_i)_{1 \leq i\leq N_T} = (t_i, \y_{t_i})_{1 \leq i\leq N_T}$, the last term in (\ref{Eqn: log likelihood General Marked Case}) should be changed to $\sum_{i=1}^{N_T} \textnormal{ln} f(\y_{t_i}; \phi | (t_j, \y_{t_j})_{1 \leq j < i})$, where $f(\cdot; \phi |  (t_j, \y_{t_j})_{1 \leq j < i}) $ corresponds to the conditional density of  the $i$th mark given  $(t_j, \y_{t_j})_{1 \leq j < i}$. This yields the log-likelihood 
\begin{equation}\label{Eqn: log likelihood General Marked Case2}
 l_g(\nu) = \int_{[0,T] \times\mathbb{X}} \ln \lambda_g(t;\nu)N_g(dt\times d\x) - \Lambda_g(T;\nu) + \sum_{i=1}^{N_T} \textnormal{ln} f(\y_{t_i}; \phi | (t_j, \y_{t_j})_{1 \leq j < i}).
\end{equation}

Some examples of the likelihood for cases where the marks are observations on a stationary process follow.
\begin{example}
The marks are i.i.d with density $f$ w.r.t some measure $c$ as in \cite{embrechts2011} and the last term in \eqref{Eqn: log likelihood General Marked Case2} becomes $\int_{[0,T]\times\mathbb{X}} \ln f(\x;\phi) N_g(dt\times d\x)$ which evaluates to $\sum_{i=1}^{N_T} \ln f(\x_i; \phi)$ and the log-likelihood is
\begin{equation} \label{Eqn: Joint log likelihood}
l_g(\nu) = \int_{[0,T]\times\mathbb{X}} \ln \lambda_g(t;\nu)N_g(dt\times d\x)- \Lambda_g(T;\nu) +  \sum_{i=1}^{N_T} \ln f(\x_i; \phi).
\end{equation}
\end{example}
\begin{example}
 More generally, the marks are observations $\x_i = \y_{t_i}$ on a stationary process $(\y_t)_{t \in \mathbb{R}_+}$ in continuous time. Then the last term in \eqref{Eqn: log likelihood General Marked Case} is the sum of the log conditional densities of $\y_{t_i}|\y_{t_{i-1}},\ldots,\y_{t_1}$ evaluated at the $\x_i$. We can write this as $\ln f(\x_1,\ldots,\x_{N_T}|t_1,\ldots,t_{N_T}; \phi)$ giving the quasi-likelihood in the form
\begin{align} \label{Eqn: Joint log likelihood Serial Dep Case}
l_g(\nu)& = \int_{[0,T]\times\mathbb{X}} \ln \lambda_g(t;\nu)N_g(dt\times d\x)- \Lambda_g(T;\nu)\nonumber\\& \quad +   \ln f(\x_1,\ldots,\x_{N_T}|t_1,\ldots,t_{N_T}; \phi)
\end{align}
where $\phi$ represents all the parameters of the joint conditional distribution including any parameters needed to model serial dependence. Note that this density depends on the event times since the specification of joint distributions for the continuous time process requires these. A simple example is when $(\y_t)_{t \in \mathbb{R}_+}$ is a stationary Gaussian process with covariance between $\y_t$ and $\y_s$ given by a function $\Gamma(s-t; \phi)$ depending on parameters $\phi$.
\end{example}
\begin{example}
	In applications to the limit order book, \cite{RichDunsmuirPeters} modelled the marks $\x_i$ at event time $t_i$ as observations on a discrete time stationary process indexed by event index $i$. The third term in \eqref{Eqn: Joint log likelihood Serial Dep Case} is replaced by $\ln f(\x_1,\ldots,\x_{N_T}; \phi)$ where $f$ now denoted the joint density for the discrete time stationary time series $\{\x_i\}$ in which actual event times are ignored and only the indices, $i$, of event times are needed to model serial dependence structure. It is not clear that this can be written as an integral with respect to $N_g(dt\times d\x)$ corresponding to the third term in \eqref{Eqn: log likelihood General Marked Case}.  Note that this leads to the objective function 
\begin{align} \label{Eqn: Joint log likelihood Discrete time Serial Dep Case}
l_g(\nu)& = \int_{[0,T]\times\mathbb{X}} \ln \lambda_g(t;\nu)N_g(dt\times d\x)- \Lambda_g(T;\nu) +   \ln f(\x_1,\ldots,\x_{N_T}; \phi)
\end{align}
to be maximised over the parameters. However this is not a formal likelihood, nor does it seem possible to define a stationary Hawkes process, as we did in Section \ref{SubSec: Existence of Stable Marked Hawkes} for the case where marks are drawn from a stationary discrete time process. Of course in the absence of serial dependence both \eqref{Eqn: Joint log likelihood Serial Dep Case} and \eqref{Eqn: Joint log likelihood Discrete time Serial Dep Case} lead to the i.i.d. version \eqref{Eqn: Joint log likelihood} considered in the literature to date.
\end{example}
When $\psi =0$ the boost is the identity so that marks do not impact the intensity. But the marks process and the event process may not be independent because the conditional distribution of marks is not free of the event times. None-the-less, the log-likelihood in \eqref{Eqn: log likelihood General Marked Case2} becomes a sum of two terms
\begin{equation*} \label{Eqn: Joint log likelihood Null Hyp}
l(\theta,\phi)=l(\theta) +  \sum_{i=1}^{N_T} \textnormal{ln} f(\y_{t_i}; \phi | (t_j, \y_{t_j})_{1 \leq j < i})
\end{equation*}
where the first term is the log-likelihood for the unmarked process $N(t)$
\begin{equation}\label{Eqn: log likelihood Unmarked Hawkes}
l(\theta) =  \int_{[0,T]} \ln \lambda(t;\theta)N(dt)- \Lambda(T;\theta)
\end{equation}
with corresponding compensator
\begin{equation*}
\Lambda(T,\theta)=\int_{0}^{T}\lambda(t;\theta)dt.
\end{equation*}
Here $N(dt)=N_g(dt,\mathbb{X})$ and $\lambda(t;\theta)=\lambda_g(t;\theta,0,\phi)$ for and $\phi \in \Phi$. The second term is the log likelihood for the marks conditional on event times.
Hence, under $H_0$, the parameters $\theta$ of the unmarked Hawkes process are decoupled from the parameters $\phi$ of the marks distribution so that these can be separately estimated. 
Note that in all forms of the quasi-likelihood, \eqref{Eqn: log likelihood General Marked Case}, \eqref{Eqn: log likelihood General Marked Case2}, \eqref{Eqn: Joint log likelihood}, \eqref{Eqn: Joint log likelihood Serial Dep Case} and \eqref{Eqn: Joint log likelihood Discrete time Serial Dep Case}, the third term involves only the parameter $\phi$ and hence the score vector with respect to the boost parameters $\psi$ does not involve the third term. Hence the score with respect to $\psi$ is the same under the null hypothesis. However, the score statistic also involves the information matrix and because the first two terms in the quasi-likelihoods do involve $\phi$ (recall that the boost function is normalized using moments of the marginal distribution of the marks), parts of Condition \ref{Cond: h function} are required so that the information matrix and hence the score test statistic are the same under the null hypothesis in all examples of likelihoods given above.
\section{The Score Test} \label{Sec: The Score Test}
Let $\nu^* = (\theta^{*}, \phi^{*},0)$ denote the true value of the combined parameters under $H_0$. Let $\hat \nu_T =(\hat \theta_T, \hat \phi_T, 0)$ where $\hat \theta_T$ is the quasi asymptotic maximum likelihood estimate, as in \citep[page 1804]{ClinetYoshida2017}, based on the likelihood \eqref{Eqn: log likelihood Unmarked Hawkes} under $H_0$ of the intensity process parameters and $\hat \phi_T$ be the MLE for the parameters of the marks density. Denote the derivatives of the log-likelihood with respect to $\nu$ as $\p_\nu l_g(\nu)$ at the parameter value $\nu$ so that $\p_\theta l_g(\nu^{*})$ and $\p_\nu l_g(\hat \nu_T)$ are evaluated at $\nu^{*}$ and $\hat \nu_T$ respectively. 
The score (or Lagrange multiplier) test statistic \citep{BreuschPaganLMtest} is defined as 
\begin{equation}\label{Eqn: Score Statistic Basic}
\hat Q_T  = \p_\nu l_g(\hat \nu_T)^\T \I(\hat \nu_T)^{-1} \p_\nu l_g(\hat \nu_T)
\end{equation}
where $\I(\nu^{*})  = \mathbb{E}_{\nu^{*}}[ \p_\theta l_g(\nu^{*}) \p_\theta l_g ( \nu^{*})^\T]$ and $\I(\hat \nu_T)$ evaluates this at the parameters, $\hat \nu_T$, estimated under $H_0$. Also \citep{BreuschPaganLMtest} the information matrix can be replaced by any matrix with the same limit in probability, for example the negative of the matrix of second derivatives of the log-likelihood, and the large sample properties of the score statistic will be the same. 
Under Condition \ref{Cond: h function} stated below on the functions $h$ defining the boost functions $g$ via \eqref{Eqn: Boost function g} the information matrix is shown in \cite{RichDunsmuirPeters} to be block diagonal which, together with $\p_\nu l_g(\hat \nu_T)= (0,0,\p_\psi l_g(\hat \nu_T))^\T$, allows simplification of \eqref{Eqn: Score Statistic Basic} to
\begin{equation}\label{Eqn: Score Statistic reduced}
\hat Q_T  = \p_\psi l_g(\hat \nu_T)^\T\I_{\psi}(\hat \nu_T)^{-1} \p_\psi l_g(\hat \nu_T)
\end{equation}
where $\I_{\psi}(\hat\nu)$ is the $r \times r$ diagonal block of $\I(\hat \nu_T)$ corresponding to $\psi$.

Because of the third term in the log-likelihood \eqref{Eqn: log likelihood General Marked Case} (and all variants \eqref{Eqn: log likelihood General Marked Case2}, \eqref{Eqn: Joint log likelihood}, \eqref{Eqn: Joint log likelihood Serial Dep Case}, \eqref{Eqn: Joint log likelihood Discrete time Serial Dep Case}) do not depend on $\psi$ it follows that
\begin{align}\label{Eqn: Score wrt general psi}
\p_\psi l_g(\nu) =  \int_{[0,T]\times\mathbb{X}}  \lambda_g(t;\nu)^{-1}\p_\psi\lambda_g(t;\nu)    N_g(dt\times d\x)-\int_{[0,T]} \p_\psi\lambda_g(t;\nu) dt
\end{align}
where
\begin{align*}
\p_\psi\lambda_g(t;\nu)=\vartheta \int_{[0,t)\times\mathbb{X}} w(t-s;\alpha) \p_\psi g(\x;\phi,\psi) N_g(ds \times d\x)
\end{align*}
and the vector of derivative of $g$ with respect to $\psi$ is
\begin{equation*}\label{Eqn: partial g wrt psi}
\partial_\psi g(\X;\phi,\psi)=\frac{1}{\E_{\phi}[h(\X;\psi)]} [ \partial_\psi h(\X;\psi)-g(\X;\phi,\psi)\E_\phi[ \partial_\psi h(\X,\psi)]].
\end{equation*}

\begin{condition} \label{Cond: h function} 
	\textbf{Conditions on boost function specification:} 
	Throughout we assume $h$, used to define the boost function $g$ in \eqref{Eqn: Boost function g}, and its first and second derivatives  with respect to $\psi$, denoted  $\partial_\psi h$ and $\p_{\psi\psi}^2 h$, satisfy the following properties: 
	\begin{enumerate}[(i)]
		\item $h(\X; 0) \equiv 1$; 
		\item $\mathbb{E}_{\phi}[h(\X;\psi)]$ and $\mathbb{E}_{\phi}[ \partial_\psi h(\X,\psi)]$ exist for all $\psi \in \Psi$, $\phi \in \Phi$;
		\item $\p_\phi \E_\phi(h(\X;\psi))|_{\nu^{*}} =0$;
		\item $\p_\phi \E_\phi[\p_\psi h(\X,\psi)]|_{\nu^{*}}$ exists for all $\phi \in \Phi$;
		\item $\cov_\phi(H(\X)) = \Omega_G(\phi)$ where $\Omega_G(\phi)$ is a finite positive definite matrix for any $\phi \in \Phi$ where $H(\X):=\partial_\psi h(\X;0)$.
	\end{enumerate}
\end{condition}
These conditions hold for all the boost functions mentioned above. Obviously, based on the properties required of $h$, $\mathbb{E}_{\phi}[g(\X;\phi,\psi)] = 1$ for all $\psi \in \Psi$, $g(\X; \phi,0) \equiv 1$,  and, letting $g(\x;\phi) =\partial_\psi g(\X;\phi,0)$,  $\mathbb{E}_{\phi}[ G(\X;\phi)]=0$. With the above specification, the null hypothesis of marks not impacting intensity is achieved by setting $\psi = 0$.  Note that $G(\X;\phi) =H(\X)-\mathbb{E}_\phi[ H(\X)]$ is a vector of dimension $r$ comprised of functions of the components of the vector mark centered at their expectations. The requirements that $\mathbb{E}_{\phi}[h(\X;\psi)]$, $\mathbb{E}_{\phi}[ \partial_\psi h(\X;\psi)]$ and $\p_\phi \E_\phi[\p_\psi h(\X,\psi)]|_{\nu^{*}}$ exist impose obvious conditions on the marginal distribution of $\X_m$. For example, if $h(\X;\psi)$ is a polynomial of degree $p$ in $\X$ then $\mathbb{E}_{\phi}[\X^p]$ needs to exist. Condition \ref{Cond: h function} parts (iii) and (iv) are required in order that the information matrix for all parameters in the full model likelihood is block diagonal allowing simplification of the score statistic defined below. For the definition of the score statistic we require the existence and positive definiteness of the covariance matrix of $g(\x;\phi)$, $\Omega_G(\phi)  = \cov_\phi(H(\X))$, as stated in (vi). Under $H_0$, the derivative of  \eqref{Eqn: Joint log likelihood} with respect to $\psi$ at any values of $\theta$, $\phi$ is

\begin{align}\label{Eqn: Score wrt general psi under H0}
\p_\psi l_g(\theta,\phi,0) =  \int_{[0,T]}  \lambda(t;\theta)^{-1}\p_\psi\lambda_g(t;\theta,\phi,0)   N(dt)-\int_{[0,T]} \p_\psi\lambda_g(t;\theta,\phi,0) dt
\end{align}
with
\[
\p_\psi\lambda_g(t;\theta,\phi,0)=\vartheta \int_{[0,t)\times\mathbb{X}} w(t-s;\alpha) G(\x;\phi) N_g(ds \times d\x).
\]
When evaluated at the estimates under the null hypothesis
\[\p_\psi l_g(\hat\nu_T) =  \int_{[0,T]}  \lambda(t;\hat\theta_T)^{-1}\p_\psi\lambda_g(t;\hat\theta_T,\hat\phi_T,0)   N(dt)-\int_{[0,T]} \p_\psi\lambda_g(t;\hat\theta_T,\hat\phi_T,0) dt
\]
and 
\begin{align*}
\p_\psi\lambda_g(t;\hat\theta_T,\hat\phi_T,0))=\hat\vartheta_T \int_{[0,t)\times\mathbb{X}} w(t-s;\hat\alpha_T) g(\x;\hat\phi_T) N_g(ds \times d\x).
\end{align*}
When evaluated at the true parameter vector, $\nu^{*}=(\theta^{*},\phi^{*},0)$ under $H_0$,  the score \eqref{Eqn: Score wrt general psi under H0} can be written as
\begin{align}\label{Eqn: Score wrt psi under H0}
\p_\psi l_g(\nu^{*}) =  \int_{[0,T]}  \lambda(t;\theta^{*})^{-1}\p_\psi\lambda_g(t;\nu^{*})  \tilde N(dt)
\end{align}
where $\tilde N(dt) = N(dt) -\lambda(t;\theta^{*})dt$
and 
\begin{align}
\I_{\psi}(\nu^*)&= \E[\int_{[0,T]}  \lambda(t;\theta^{*})^{-2}(\p_\psi\lambda_g(t;\nu^{*}))^{\otimes 2} N(dt)],\label{Eqn: Ipsipsi under H0, ver 3}
\end{align}
where $x^{\otimes 2} = x.x^\T$. Noting that the expectation required to evaluate \eqref{Eqn: Ipsipsi under H0, ver 3} 
is not computable in closed form, we suggest empirical evaluation replacing the expectation by the time average over events and using the estimate $\hat \nu_T$ to get
\begin{equation}
\hat \I_{\psi} = \int_{[0,T]}  \lambda(t;
\hat\theta_T)^{-2}(\p_\psi\lambda_g(t;\hat\nu_T))^{\otimes 2} N(dt)\label{Eqn: Estimated Ipsipsi under H0, ver 3}.
\end{equation}
We show that this empirical estimate has the same asymptotic limit as \eqref{Eqn: Ipsipsi under H0, ver 3} when scaled by $T$. 
Using these estimates in the definition \eqref{Eqn: Score Statistic reduced}, the score statistic can be implemented in practice as
\begin{equation}\label{Eqn: Score Statistic implemented form}
\hat Q_T  = \p_\psi l_g(\hat \nu_T)^\T\hat\I_{\psi}^{-1} \p_\psi l_g(\hat \nu_T)
\end{equation}
where $\p_\psi l_g(\hat \nu_T)$ is defined above and $\hat\I_{\psi}$ is given by \eqref{Eqn: Estimated Ipsipsi under H0, ver 3}.

\section{Asymptotic Distribution of the Score Statistic} \label{Sec: Asymptotic Dist Score Test}

To prove that the score statistic $\hat Q_T$ has a large sample chi-squared distribution under the null hypothesis conditions are required on the intensity process for the unboosted process. The extra conditions are those required for convergence of the quasi MLE for Hawkes processes under $H_0$, for which the intensity does not depend on marks. Because we adapt the proofs of \citep[Theorems 3.9 and 3.11]{ClinetYoshida2017} to the score statistic we re-state their conditions [A1], [A2], [A3] and [A4] here. These generalize Conditions A, B and C of \cite{ogata1978asymptotic} applied to the intensity process defined in \eqref{Eqn: Intensity Process Null Hyp} for the unmarked process. \cite{ogata1978asymptotic} provided the first consistency and asymptotic normality results for the unmarked Hawkes process and verified that his conditions apply to the exponential decay function $w(t;\alpha)$. \cite{ClinetYoshida2017} give conditions for the convergence of moments of the quasi MLE and verify them for the exponential decay function case. As far as we are aware there has been no published verification of the conditions of \cite{ogata1978asymptotic} or \cite{ClinetYoshida2017} for the power law decay function.

\begin{condition} \label{Cond: Conditions on Quasi MLE under null hypothesis} \textbf{Conditions on the intensity process under $H_0: \psi = 0$.}  For clarity, these are restated from \cite{ClinetYoshida2017} using notation of this paper and as relevant to the Hawkes process. These conditions refer to the intensity process defined in \eqref{Eqn: Intensity Process Null Hyp}. Recall that $\theta^{*}$ refers to the true parameter defining the intensity process under $H_0$.
	
	\begin{description}
		\item[A1] The mapping $\lambda: \Omega \times \mathbb{R}_{+}\times \Theta \to \mathbb{R}_{+}$ is $\F \otimes \mathbf{B}(\mathbb{R}_{+})\otimes \mathbf{B}(\Theta)$-measurable. Moreover, almost surely:
		\begin{enumerate}[(i)]
			\item for any $\theta \in \Theta$, $s \to \lambda(s,\theta)$ is left continuous;
			\item for any $s \in \Rp$, $\theta \to \lambda(s,\theta)$ is in $C^3(\Theta)$ and admits a continuous extension to $\bar \Theta$.
		\end{enumerate}
		\item[A2] The intensity process $\lambda$ and its derivatives satisfy,
			 for any $p >1$, $$\sup_{t \in \Rp} \sum_{i=0}^{3}\|\sup_{\theta \in \Theta }|\p^i_\theta \lambda(t,\theta)|\|_p < \infty.$$
		\item[A3] For a Borel space $(E, \mathbf{B}(E))$ let $C_b(E,\R)$ be the set of continuous, bounded functions from $E$ to $\R$. For any $\theta \in \Theta$ the triplet $(\lambda(\cdot, \theta^*),\lambda(\cdot, \theta), \p_\theta \lambda(\cdot,\theta))$ is ergodic in the sense that there exists a mapping $\pi: C_b(E,\R)\times \Theta \to \R$ such that for any  $(\xi, \theta) \in C_b(E,\R)\times \Theta$, 
		\[
		\frac{1}{T}\int_{0}^{T}\xi(\lambda(s,\theta^*),\lambda(s,\theta),\p_\theta \lambda(s,\theta))ds \to^ \mathbb{P} \pi(\xi,\theta).
		\]
		
		\item[A4] Define 
		\[
		\mathbb{Y}_T(\theta)=\frac{1}{T}(l_T(\theta)-l_T(\theta^{*}),
		\] which is shown in \citep[Lemma 3.10]{ClinetYoshida2017} to satisfy
		\[
		\sup_{\theta \in \Theta }|
		\mathbb{Y}_T(\theta)-\mathbb{Y}(\theta)| \to^{\mathbb{P}}0
		\]
		and $\mathbb{Y}(\theta)$ is the ergodic limit of $\mathbb{Y}_T(\theta)$ as defined in \citep[p. 1807]{ClinetYoshida2017}.
		Assume, for asymptotic identifiability, that for any $\theta \in \bar \Theta -\{\theta^{*}\}$, $\mathbb{Y}(\theta)\ne 0$.
	\end{description}
\end{condition}	

Under Condition \ref{Cond: Conditions on Quasi MLE under null hypothesis}: [A1] to [A4], \cite{ClinetYoshida2017} show (Theorem 3.9) that any asymptotic QMLE $\hat \theta_T$ is consistent, $\hat \theta_T \to^{\mathbb{P}} \theta^{*}$, and (Theorem 3.11) asymptotically normal $\sqrt{T}(\hat \theta_T - \theta^{*}) \to^{d} \Gamma^{-\frac{1}{2}}\zeta$ where $\zeta$ has a standard multivariate normal distribution and $\Gamma$ is the asymptotic information matrix, assumed to be positive definite. Additionally they prove that $\Gamma$ satisfies
\[
\sup_{\theta \in V_T}|T^{-1}\p_\theta^2 l_T(\theta)+\Gamma|\to^{P} 0,
\]
where $V_T$ is a ball shrinking to $\theta^*$.  

As noted above these conditions are met for the (multivariate) exponential decay Hawkes process without marks as shown in 
\citep[Section 4]{ClinetYoshida2017} assuming each element of $\theta = (\eta,\vartheta,\alpha)$ belongs to finite closed intervals of $\R$. For example, for the exponential decay function $w(s;\alpha)= \alpha \exp(-\alpha s)$, $K=3$ and we assume that $0<\underline{\eta}\le \eta \le \bar{\eta}<\infty$, $0<\underline{\vartheta}\le \vartheta \le \bar{\vartheta}<\infty$, $0<\underline{\alpha}\le \alpha \le \bar{\alpha}<\infty$ so that $\Theta$ is a finite dimensional relatively compact open subset of $\R^3$. 

In order to establish the asymptotic distribution of the score vector with respect to $\psi$, Condition \ref{Cond: Conditions on Quasi MLE under null hypothesis} A.2 needs to be extended to accommodate the contribution to the score vector from the marks as follows.
\begin{condition}\label{Cond: Condition on sup lambda_g}
	For $p = (dim(\Theta)+1)\vee 4$, where $x \vee y = \max(x,y)$, under $H_0$ and with $\phi$ fixed at $\phi^*$, assume
	\begin{equation}\label{Eqn: Condition on uniform lambda_g}
	\sup_{t \in \Rp} \sum_{i=0}^{2}\|\sup_{\theta \in \Theta }|\p^i_\theta(\p_\psi\lambda_g(t;\theta, \phi, \psi)|_{(\theta, \phi^*,0)})|\|_p < \infty.
	\end{equation}
\end{condition}
\begin{lem} \label{Lem 1}
	Condition \ref{Cond: Condition on sup lambda_g} is satisfied for the exponential decay function model (for which $dim(\Theta) = 3$) and stationary ergodic marks for which $\E_{\phi^*}[|G(\X)|^4]<\infty$. 
\end{lem}
The proof of Lemma \ref{Lem 1} is given in Appendix \ref{Sec: Proof Lemma 1}.

The final condition concerns the marks. In the remainder some additional notation is helpful. Recall that, under the null hypothesis the true parameter vector is $\nu^* = (\theta^*,\phi^*, 0)$ and the maximum likelihood estimates are $\hat \nu=(\hat \theta_T, \hat \phi_T,0)$. Denote $\mu_H(\phi) =E_{\phi}[H(\X)]$ and if evaluated at $\phi^*$ put $\mu_H = \mu_H(\phi^*)$ and if evaluated at $\hat \phi_T$ put $\hat \mu_H$. We also use the same notation for any consistent estimate of $\mu_H$ such as $\hat \mu_H = \bar H(\X)$, the vector of sample means of components. We let $G(\X) = H(\X)-\mu_H$ at the true value and $\hat G(X) = H(\X)-\hat \mu_H$. 
\begin{condition}\label{Cond: Marks Stat Ergodic}
The marks are from a stationary ergodic process with $\E_{\phi^*}[|G(\X)|^4]<\infty$ and $\hat \mu_H \to^P \mu_H$ as $T \to \infty$.
\end{condition}

\noindent Note that $\hat \mu_H \to^P \mu_H$ holds for either the sample mean estimate (using ergodicity of $\X$), the parametric form, $\hat \mu_H = \mu_H(\hat \phi_T)$ (using consistency of the maximum likelihood estimates $\hat \phi_T$ under appropriate regularity conditions on $f_t(\x;\phi)$) or any other consistent estimates of $\phi$ such as using method of moments.

We now state the main result.  
\begin{thm} \label{Thm: Score Statistic Asymptotic Chi-squared}
	Assume Conditions \ref{Cond: h function},  \ref{Cond: Conditions on Quasi MLE under null hypothesis}, \ref{Cond: Condition on sup lambda_g} and \ref{Cond: Marks Stat Ergodic}. Under $H_0$, the score statistic defined in \eqref{Eqn: Score Statistic reduced} with information matrix $\I_{\psi}(\hat \nu_T)$ estimated by $\hat \I_{\psi}$  defined in \eqref{Eqn: Estimated Ipsipsi under H0, ver 3} satisfies
	\begin{equation} \label{Eqn: Score Stat Asymptotic Chi-squared}
	\hat Q_T \overset{\textrm{d}}{\longrightarrow}  \chi_{(r)}\quad \textrm{as } T\to\infty, \quad r = dim(\psi).
	\end{equation}
\end{thm}
The proof is given in Appendix \ref{Sec: Proof Thm 1}.
\section{Local power} \label{Sec: Local Power}
We now investigate what happens to the distribution of the score statistic when $H_0$ fails, that is when the mark process impacts the distribution of the jump times of the point process. We adopt the local power approach, which consists in considering the sequence of local alternatives $H_1^T : \psi_T^* = \gamma^*/\sqrt{T}$ for some unkown $\gamma^*$. We therefore assume that the marks weakly impact the distribution of the jump times (with a magnitude of order $1/\sqrt{T}$), so that for a given $T >0$, the associated counting process is nearly a pure Hawkes process. Our goal is to derive the asymptotic distribution of the score statistic under the local alternatives $H_1^T$. Following Proposition \ref{propExistence}, we thus assume that we observe a sequence of marked Hawkes processes $N_g^T$, all defined on (and adapted to) the same probability space $(\Omega, \calf, \mathbb{P})$. Note that we adopt the notation $N_g^T$ because, in contrast with the null hypothesis, the point process now depends on $T$. Moreover, we assume that all the marked Hawkes processes indexed by $T$ are generated by the random measure $\bar{N}_{g}$ on $\mathbb{R}^2 \times \mathbb{X}$, such that the normalized boost function of $N_g^T$ is $g(., \phi^*, \psi_T^*)$, that is, for any $t \in \mathbb{R}_+$, $N_{g}^T$ admits the following stochastic intensity: 
$$\lambda_g^{T}(t; \theta^*,\phi^*,\psi_T^*) = \eta^* + \vartheta^* \int_{(-\infty,t) \times \mathbb{X} } w(t-s;\alpha^*) g(\x; \phi^*,\psi_T^*)N_g^T(ds \times d\x),$$
for some unkown parameter $\nu_T^* = (\theta^*, \phi^*, \psi_T^*)$. The expression of the score statistic is naturally adapted to    
\begin{equation} \label{scoreStatisticNewDef}
\p_\psi l_g^T(\hat \nu_T)^\T\I_{\psi}(\hat \nu_T)^{-1} \p_\psi l_g^T(\hat \nu_T),
\end{equation}
where $l_g^T$ admits the same expression as in (\ref{Eqn: log likelihood Unmarked Hawkes}), replacing the pure Hawkes process $N(dt)$ by the counting process $N^T(dt) = N_g^T(dt,\mathbb{X})$. Similarly, in (\ref{scoreStatisticNewDef}), $\hat{\nu}_T = (\hat{\theta}_T, \hat{\phi}_T, 0)$, where $\hat{\theta}_T$ is one maximizer of $l_g^T$ in the interior of $\Theta$, and $\hat{\phi}_T$ is a consistent estimator of $\phi^*$. As stated in Theorem \ref{thm LocalPower} below, it turns out that under $H_1^T$, $\hat{Q}_T$ tends to a non central chi-squared distribution, whose non-centrality parameter depends on $\gamma^*$ and on the inverse of the Fisher information matrix (at point $\psi = 0$), $\Omega$. In order to ensure the convergence of $\hat{Q}_T$, we make the following assumptions.

\begin{condition}\label{conditionLocalPower}
	For $p = (dim(\Theta) + 1) \vee 4$, we assume the existence of $\epsilon >0$ such that, defining $\mathcal{U} = \Theta \times \{\phi^*\}\times \mathcal{B}(0,\epsilon)$ where $\mathcal{B}(0,\epsilon)$ is the open ball of radius $\epsilon$,
	$$ \sup_{T \in \Rp}\sup_{t \in [0,T]} \sum_{i=0}^3 \E \left[ \sup_{\nu \in \mathcal{U}} |\partial_\theta^i \lambda_g^{T}(t;\nu)|^p\right] < + \infty.$$
	Moreover,
	$$ \sup_{T \in \Rp}\sup_{t \in [0,T]} \sum_{i=0}^2 \E \left[ \sup_{\nu \in \mathcal{U}} |\partial_\theta^i \partial_\psi \lambda_g^{T}(t;\nu)|^p\right] < + \infty.$$
	Moreover, assume that there exists $\epsilon >0$ such that
	\begin{equation} \label{momentLocalPower}
	\E \sup_{\psi \in \mathcal{B}(0,\epsilon) } \left| \partial_\psi g(\x;\phi^*,\psi)\right|^p < +\infty.
	\end{equation}
	Finally, for $q \in \{1,2\}$, defining $\mathcal{A} = \{\alpha | \exists (\eta, \vartheta) \textnormal{ s.t. } (\eta,\vartheta,\alpha) \in \Theta \}$, we assume the existence of $\bar{w}$ such that for any $\alpha \in \mathcal{A}$, for any $t \geq 0$, $w(t;\alpha) \leq \bar{w}(t),$ and 
	\begin{equation} \label{uniform w localpower}
	\int_0^{+\infty} \bar{w}(t)^q  dt< \infty.
	\end{equation} 
	
\end{condition}

Condition \eqref{uniform w localpower} is satisfied for the exponential decay function under the conditions stated above for $\alpha$. For suitable choice of a compact parameter space for the power law decay function, a two parameter family of decay functions, the condition is also satisfied without placing undue restrictions on the parameter space.

\begin{lem}\label{lem exponentialHawkes localPower}
	Condition \ref{conditionLocalPower} is satisfied for the exponential kernel case ($dim(\Theta)=3$) and for stationary marks satisfying (\ref{momentLocalPower}).  
\end{lem}

\begin{proof}
	The proof follows exactly the same path as that of Lemma \ref{Lem 1}, replacing the fourth order moment condition on $G(\x)$ by the local uniform condition (\ref{momentLocalPower}).
\end{proof}

Before we state the main result of this section, we define 
\begin{eqnarray*}
\Omega = \mathbb{P}-\lim_{T \to + \infty} T^{-1}\I_{\psi}(\nu^*),
\end{eqnarray*}
where we recall that $\I_{\psi}(\nu^*)$ was defined in (\ref{Eqn: Ipsipsi under H0, ver 3}) and corresponds to the Fisher information matrix associated to $\psi$, at point $\psi = 0$, under the null hypothesis. We prove that such a limit exists in Appendix B (Lemma \ref{Lem 3}). We can now state the following theorem.
\begin{thm}\label{thm LocalPower}
	Assume Conditions 1,2,3,4 and 5. Under $H_1^T : \psi_T^* = \gamma^*/\sqrt{T}$, we have 
	$$ \hat{Q}_T \to^d  \chi^2(\Omega^{1/2}\gamma^*),$$
	where $\chi^2(\Omega^{1/2}\gamma^*) \sim \|Z\|^2$ with $Z \sim \mathcal{N}(\Omega^{1/2}\gamma^*,1)$.
\end{thm}
The proof is in Appendix \ref{Sec: Proof Thm 2}.
\section{Conclusions and Future Extensions} \label{Sec: Conclusions etc}

In this paper we have derived the asymptotic distribution of the score test proposed for determining if marks have an impact on the intensity of a single Hawkes process. Quite general boost functions can be formulated in this setting. We prove that the asymptotic distribution under the null hypothesis that there is no impact of the proposed marks on the intensity process is the usual chi-squared distribution with degrees of freedom equal to the number of parameters specified for the marks boost function. 
These asymptotic results rely heavily on the large sample results for quasi-likelihood estimation of multivariate unmarked Hawked process considered in \cite{ClinetYoshida2017}. In addition to their assumptions on the null hypothesis model specification and parameters, because the score test involves functions of the marks, one additional assumption (Condition \ref{Cond: Condition on sup lambda_g})  is required, and this is shown to hold in the exponential decay case (see Lemma \ref{Lem 1}).

The marks process can be quite general and includes marks obtained from observations on a continuous time vector valued process in which there is serial dependence as well as dependence between components of the mark vector. The main requirement is that the marks have finite fourth moment.  

For local power computations, we have also derived the non-central chi-squared limiting distribution for the score test statistic under a sequence of local alternatives with the boost parameter converging to the null hypothesis value at rate $T^{-1/2}$. 

Establishing consistency of the score test requires a proof that the power tends to unity for any value of $\psi \ne 0$. However, establishing this rigorously requires proving the ergodicity of the point process along with substantial extensions to existing asymptotic theory for likelihood estimation in marked Hawkes processes. The main technical challenge for establishing this is showing that the asymptotic score w.r.t. $\phi$ is non-degenerate. Here a major difficulty arises because the existence of multiple stationary values in the limiting likelihood function of $(\theta,\phi)$ when $\psi \ne 0$ cannot be ruled out easily.

Crucial to establishing the conditions required for the results of \cite{ClinetYoshida2017} as well as our additional Condition \ref{Cond: Condition on sup lambda_g} is the Markovian nature of the Hawkes intensity process with an exponential decay function. Extension to decay functions which are linear combinations of exponential kernels retain the Markov property and so extension of above results should be straightforward. For other kernels, such as the power law decay function, the Markov property does not hold and hence extension of our results would require substantial and fundamental theory to extend known results in the literature firstly in the unmarked Hawkes processes and secondly in the marked case.
Because \cite{ClinetYoshida2017} also establish the required asymptotic theory of likelihood estimation for a multivariate unmarked Hawkes process and because the form of the score statistic for a marked multivariate Hawkes process is of the same basic form as for the univariate Hawkes process the results of this paper should readily extend to the multivariate case and could be the topic of future research.
\appendix
\section{Proof of Lemma \ref{Lem 1}}\label{Sec: Proof Lemma 1}
For any $c \in \R^r$ we denote the linear combinations $G_c(\X)= c^\T G(\X)$ and similarly for $\hat G_c(\X)$. We use the notation $N_g^0$ for the point process generated under $H_0$. This point process has event intensity identical to that of $N$ defined in \eqref{Eqn: Intensity Process Null Hyp}. Marks are observed at the event times of this process but do not impact the intensity of it.
For the exponential decay specification, since $dim(\Theta) = 3$, we need to show Condition \ref{Cond: Condition on sup lambda_g} for $p = 4$.
\begin{equation*}\label{Eqn: Extended A2}
\sup_{t \in \Rp} \E[\sup_{\theta \in \Theta }|\p^i_\theta \{\vartheta\int_{[0,t)\times\mathbb{X}}  w(t-s;\alpha)G_c(\x)N_g^0(ds \times d\x)\}|^p] < \infty
\end{equation*}
for $i = 0, 1, 2$. Notice that only the derivatives with respect to $\vartheta$ and $\alpha$ are required. These derivatives are linear combinations of terms of the form 
\begin{equation*}
\vartheta^{k}\int_{[0,t)\times\mathbb{X}}  \p_\alpha^i w(t-s;\alpha)G_c(\x)N_g^0(ds \times d\x)
\end{equation*}
for $i = 0,1,2$ and $k=0,1$, and with $w(t-s;\alpha)= e^{-\alpha (t-s)}$. Since $\Theta$ is bounded, we consider the integrals which are finite combinations of terms of the form
\begin{equation*}
\int_{[0,t)\times\mathbb{X}} (t-s)^i e^{-\alpha(t-s)}G_c(\x)N_g^0(ds \times d\x) \leq \int_{[0,t)\times\mathbb{X}} (t-s)^i e^{-\underline{\alpha}(t-s)}G_c(\x)N_g^0(ds \times d\x)
\end{equation*}
for $i=0,1,2$, and where $0 < \underline{\alpha} = \inf \{\alpha | \exists (\eta,\vartheta), (\eta,\alpha,\vartheta)  \in \Theta \} $. Therefore, we need to show to conclude the proof that 
\begin{equation*}\label{Eqn: Lemma 4 key inequality}
\sup_{t \in \Rp} \E|\int_{[0,t)\times\mathbb{X}}  (t-s)^i e^{-\underline\alpha(t-s)}G_c(\x)N_g^0(ds \times d\x)|^p<\infty, \quad i=0,1,2.
\end{equation*}
where $p=4$. Let $f_{i,t}(s)=(t-s)^i \exp(-\underline\alpha(t-s))$. For any $t \geq 0$, we have 
\begin{align*}
&\E[(\int_{[0,t)\times \mathbb{X}}f_{i,t}(s)G_c(\x)N_g^0(ds \times d\x) )^4]\\
&\le C \E[(\int_{[0,t)\times \mathbb{X}}f_{i,t}(s)|G_c(\x)|\tilde N_g^0(ds \times d\x))^4]\\
& \quad+C\E[(\int_{[0,t)\times \mathbb{X}}f_{i,t}(s)|G_c(\x)|F_s(d\x) \lambda(s;\theta^*)ds )^4]
\end{align*}
for some finite constant $C$ and where the compensator of $N_g^0(ds \times d\x)$ is $ \lambda(s;\theta^*) F_s(d\x)ds$ where $F_s(d\x)$ is the conditional distribution of $\y_s$ with respect to $\calf_{s-}^\y$. 

First define the probability measure $\mu(ds) = (\int_0^t f_{i,u} du)^{-1} f_{i,s}ds$ on $[0,t]$, and apply Jensen's inequality to the second term to get
\begin{align*}
&\E[(\int_{[0,t)\times \mathbb{X}}f_{i,t}(s)|G_c(\x)|F_s(d\x)\lambda(s;\theta^*)ds)^4]\\
&= (\int_{[0,t)}f_{i,t}(s)ds)^4 \E[(\int_{[0,t)\times \mathbb{X}} |G_c(\x)|F_s(d\x)\lambda(s;\theta^*) \mu(ds))^4]\\ 
&\leq (\int_{[0,t)}f_{i,t}(s)ds)^3 \E[\int_{[0,t)}  f_{i,t}(s) (\int_{\mathbb{X}}|G_c(\x)|F_s(d\x))^4 \lambda(s;\theta^*)^4ds]\\
&\leq (\int_{[0,t)}f_{i,t}(s)ds)^3 \E[\int_{[0,t)}f_{i,t}(s)\lambda(s;\theta^*)^4 \E[\E[|G_c(\x)| | \calf_{s-}^\y]^4]ds]\\
&\leq C, 
\end{align*}
Where we have used the independence of $\y$ and $N_g^0$, the fact that $\E[\E[|G_c(\y_s)| | \calf_{s-}^\y]^4]] < \E[|G_c(\y_s)|^4] <  K$ for some constant $K >0$, and $\sup_{t \in \Rp}\E[\int_{[0,t)}f_{i,t}(s)\lambda(s;\theta^*)^4ds]<\infty$ by \citep[Lemma A.5]{ClinetYoshida2017}. Consider now the first expected value. Using Davis-Burkholder-Gundy inequality we have arguing similarly to \citep[Lemma A.2]{ClinetYoshida2017}, for some constant $C<\infty$ not necessarily the same as above,
\begin{align*}
&\E[(\int_{[0,t)\times \mathbb{X}}f_{i,t}(s)|G_c(\x)|\tilde N_g^0(ds \times d\x))^4]\\
& \le C \E[(\int_{[0,t)\times \mathbb{X}}f_{i,t}(s)^2|G_c(\x)|^2 N_g^0(ds \times d\x))^2]\\
& \le 2 C \E[(\int_{[0,t)\times \mathbb{X}}f_{i,t}(s)^2|G_c(\x)|^2 \tilde N_g^0(ds \times d\x))^2] \\
& \quad + 2 C \E[(\int_{[0,t)\times \mathbb{X}}f_{i,t}(s)^2|G_c(\x)|^2\lambda(s;\theta^*)F_s(d\x)ds)^2].
\end{align*}
Similarly to the previous argument the second term is uniformly bounded because $$\sup_{t \in \Rp}\E[(\int_{[0,t)}f_{i,t}(s)^2\lambda(s;\theta^*)^2ds)^2]<\infty$$ by \citep[Lemma A.5]{ClinetYoshida2017} and $\E[\E[|G(\X)^2|\calf_{s-}^\y]^2] < \E|G_c(\y_s)|^4 < K$ for some constant $K >0$. For the first term we have
\begin{align*}
&\E[(\int_{[0,t)\times \mathbb{X}}f_{i,t}(s)^2|G_c(\x)|^2 \tilde N_g^0(ds \times d\x))^2] \\
& = \E[\int_{[0,t)\times \mathbb{X}}f_{i,t}(s)^4|G_c(\x)|^4 \lambda(s;\theta^*)F_s(d\x)ds]\\
& = \E[\int_0^t f_{i,t}(s)^4\lambda(s;\theta^*)\E|G_c(\y_s)|^4ds],
\end{align*}
where we have used the independence of $\y$ and $N_g$. Now, $\E|G_c(\y_s)|^4 < \infty$ and, once more by \citep[Lemma A.5]{ClinetYoshida2017} we have 
$$\sup_{t \in \Rp}\E[\int_{[0,t)}f_{i,t}(s)^4\lambda(s;\theta^*)ds]<\infty$$
which completes the proof.

\section{Proof of Theorem 1} \label{Sec: Proof Thm 1}
Define, for any fixed $c$,
\begin{equation}\label{Eqn: U(t) definition}
U(t;\theta,\phi) = c^\T \p_\psi \lambda_g(t;\theta,\phi,0) = \vartheta \int_{[0,t)\times\mathbb{X}} w(t-s;\alpha) c^\T G(\x;\phi) N_g(ds \times d\x)
\end{equation}
This notation is used repeatedly in the proof of the theorem as well as the lemmas used. The proof follows somewhat closely that of \cite{ClinetYoshida2017}. We first consider the normalized process corresponding to \eqref{Eqn: Score wrt psi under H0} and for any non zero vector of constants $c\in \R^r$ define the process in  $u \in [0,1]$
\begin{align}\label{Eqn: SuT process}
S_u^T  &=  \frac{1}{\sqrt{T}}\int_{[0,uT]}  \lambda(t;\theta^{*})^{-1}c^\T\p_\psi\lambda_g(t;\nu^{*})  \tilde N(dt)\nonumber\\
&= \frac{1}{\sqrt{T}}\int_{[0,uT]}  \lambda(t;\theta^{*})^{-1} U(t;\theta^*, \phi^*)\tilde N(dt).
\end{align}
Note that $S_1^T=\frac{1}{\sqrt{T}}c^\T\p_\psi l_g(\nu^{*})$. Similarly to \cite{ClinetYoshida2017}, we establish a functional CLT when $T \to \infty$.

The proof of this theorem proceeds via several lemmas. Convergence throughout is with $T\to\infty$. The first lemma is concerned with the ergodic properties of  $U(t;\theta,\phi)$ defined in \eqref{Eqn: U(t) definition} when $\phi=\phi^*$, is fixed at the true value in which case we further abbreviate notation to $U(t;\theta)= U(t;\theta,\phi^*)$.

\begin{lem} \label{Lem U ergodicity} 
	There exists a stationary Hawkes point process $N^{\infty}$ on the original probability space $(\Omega, \calf,\mathbb{P})$, adapted to $\calf_t$ and defined on $\mathbb{R}$ such that: (i) $N^{\infty}$ and $\y$ are independent. (ii) the stochastic intensity of $N^{\infty}$ admits the representation $ \lambda^{\infty}(t) = \eta + \vartheta^*\int_{(-\infty,t)} w(t-s;\alpha^*)N^{\infty}(ds)$.
	Moreover, let us define $N_g^{\infty}$ as the marked point process which jumps at points of the form $(t_i^{\infty}, \y_{t_i^\infty})$ where $t_i^\infty$ are the jump times of $N^\infty$. Accordingly, we define 
	$$ U^{\infty}(t;\theta) = \vartheta \int_{(-\infty,t)\times\mathbb{X}} w(t-s;\alpha) G_{c}(\x) N_g^\infty(ds \times d\x).$$
	Then, the joint process $(\lambda^\infty, U^{\infty}(.;\theta^*))$ is stationary ergodic. Finally we have the convergence
	\begin{equation}\label{Eqn: deviation L1}
	\E | \lambda(t,\theta^*) -  \lambda^{\infty}(t) | + \E | U(t;\theta^*) - U^\infty(t;\theta^*)| \to 0, t \to +\infty.
	\end{equation}
\end{lem}
\begin{proof} The existence of $N^\infty$ along with properties (i) and (ii) are direct consequences of the independence of $N$ and $\y$, along with Proposition 4.4 (i) from \cite{ClinetYoshida2017}. Next, since $\y$ is ergodic by assumption, the process of jumps $N^\infty$ is stationary ergodic by assumption, and since both processes are independent from each other, the joint process $(N^\infty,\y)$ is stationary ergodic as well. Since for any $t \in \mathbb{R}$, $(\lambda^\infty(t), U^\infty(t,\theta^*))$ admits a stationary representation and given the form of $(\lambda^\infty(t), U^\infty(t,\theta^*))_{t \in \mathbb{R}}$, we can deduce that they are also ergodic by Lemma 10.5 from \cite{kallenberg2006foundations}. Finally, we show \eqref{Eqn: deviation L1}. We first deal with the convergence of $f(t):=\E | \lambda(t,\theta^*) -  \lambda^{\infty}(t) |$ to $0$. Defining $r(t) = \E\int_{(-\infty,0]} w(t-s;\alpha^*)\lambda^\infty(s) ds$, and following the same reasoning as for the proof of Proposition 4.4 (iii) in \cite{ClinetYoshida2017}, some algebraic manipulations easily lead to the inequality
	$$ f(t) \leq r(t) + \vartheta^* w(.;\alpha^*) \ast f(t), \textnormal{ } t \geq 0$$
	where for two functions $a$ and $b$, and $t \in \mathbb{R}_+$, $a \ast b (t) = \int_0^t a(t-s)b(s)ds$ whenever the integral is well-defined. Iterating the above equation, we get for any $n \in \mathbb{N}$
	$$ f(t) \leq \sum_{k=0}^n \vartheta^{*k} w(.;\alpha^*)^{\ast k} \ast r(t) + \vartheta^{*n+1} w(.;\alpha^*)^{\ast(n+1)}\ast f(t).$$
	Using the fact that $\int w(.;\alpha^*) = 1$, $\vartheta^* < 1$ and using Young's convolution inequality we easily deduce that the second term tends to 0 as $n$ tends to infinity, so that $f$ is dominated by $R \ast r$ where $ R := \sum_{k=0}^{+\infty} \vartheta^{*k} w(.;\alpha^*)^{\ast k}$. Note that $R$ is finite and integrable since $\int_0^{+\infty} R(s)ds  \leq 1/(1-\vartheta)$. We first prove that $r(t) \to 0$. To do so, note that $r(t) = \E\left[\lambda^\infty(0)\right] \int_t^{+\infty} w(u;\alpha^*)du \to 0$ since $w(.;\alpha^*)$ is integrable. Now, since $R \ast r(t) = \int_0^t R(s)r(t-s)ds$, and $R(s)r(t-s)$ is dominated by $\textrm{sup}_{u \in \mathbb{R}_+} r(u) R(s)$ which is integrable, we conclude by the dominated convergence theorem that $f(t) \leq R \ast r(t) \to 0$. Finally, we prove that $g(t) := \E | U(t;\theta^*) - U^\infty(t;\theta^*)| \to 0$. Let $N_g^0(ds \times d\x)$ denote the point process for  $s \in \R$ and under the null hypothesis. We have
	\begin{align*}
	g(t) &\leq  \E \left| \int_{(0,t) \times \mathbb{X}}w(t-s;\alpha^*) G_c(\x) (N_g^0 - N_g^\infty)(ds \times d\x)\right|\\&\quad \quad + \E\left|\int_{(-\infty,0) \times \mathbb{X}} w(t-s;\alpha^*)G_c(\x)N_g^\infty(ds \times d\x)\right|\\
	&\leq\E \int_{(0,t) \times \mathbb{X}}w(t-s;\alpha^*) |G_c(\x)||N_g^0 - N_g^\infty|(ds \times d\x) \\&\quad \quad+ \E \int_{(-\infty,0) \times \mathbb{X}} w(t-s;\alpha^*)|G_c(\x)|N_g^\infty(ds \times d\x)\\
	&\leq \E|G_c(\x)| \int_{(0,t)} w(t-s;\alpha^*) f(s)ds\\&\quad \quad + \E|G_c(\x)|\E\left[\lambda^\infty(0)\right] \underbrace{\int_{t}^{+\infty} w(u;\alpha^*)du}_{\to 0}.
	\end{align*}
	Since $\int_{(0,t)} w(t-s;\alpha^*) f(s)ds = \int_{(0,t)} w(s;\alpha^*) f(t-s)ds$ and $f(t)\to 0$, we have, again, by application of the dominated convergence theorem that $g(t) \to 0$. 
\end{proof}
\begin{lem} \label{Lem 3} $S_u^T$ defined in \eqref{Eqn: SuT process} satisfies
	\begin{equation*}
	(S_u^T)_{u \in [0,1]} \to^d \Omega^{1/2} (W_u)_{u \in [0,1] }
	\end{equation*}
	where $W$ is standard Brownian motion (and convergence is in the Skorokhod space $\mathbf{D}([0,1])$) and $\Omega$ is a positive definite matrix.
\end{lem}
\begin{proof}
	Similarly to \citep[proof of Lemma 3.13]{ClinetYoshida2017} we first show that
	\begin{equation*}
	\langle S^T,S^T \rangle_u  =\frac{1}{T} \int_{[0,uT]}  \lambda(t;\theta^{*})^{-1}U(t;\theta^*, \phi^*)^2 dt
	\end{equation*}
	converges in probability to $u c^\T\Omega c$. Introducing $\lambda^\infty$, $U^\infty$ as in Lemma \ref{Lem U ergodicity}, we need to show that 
	\begin{equation} \label{eqn: bracketErgodicity}
	\frac{1}{T} \int_{[0,uT]}  \left\{\lambda(t;\theta^{*})^{-1}U(t;\theta^*, \phi^*)^2 - \lambda^\infty(t)^{-1}U^\infty(t;\theta^*)^2 \right\}dt \to^P 0.
	\end{equation}
	Using the boundedness of $\lambda(t;\theta^*)^{-1}$ and $\lambda^{\infty}(t)^{-1}$, we have the domination 
	\begin{align*}
	A_t := &\left|\frac{U(t;\theta^*, \phi^*)^2}{\lambda(t;\theta^{*})} - \frac{U^\infty(t;\theta^*)^2}{\lambda^\infty(t)}\right|\\& \leq K\left|U(t;\theta^*, \phi^*)^2 - U^\infty(t;\theta^*)^2\right| + KU^\infty(t;\theta^*)^2\left|\lambda(t;\theta^*)- \lambda^{\infty}(t)\right|,
	\end{align*}
	for some constant $K>0$. By Lemma \ref{Lem U ergodicity}, we thus have $A_t \to^P 0$. Moreover, since by Condition 3, $U(t;\theta^*, \phi^*)$ and $U^\infty(t;\theta^*)$ are $\mathbb{L}^{2+\epsilon}$ bounded for some $\epsilon >0$, and $\lambda(t;\theta^*)$ and $\lambda^{\infty}(t)$ are $\mathbb{L}^p$ bounded for any $p>1$, we deduce that $\E |A_t| \to 0$. This, in turn, easily implies that $\E|T^{-1} \int_0^{uT} A_t dt| \to 0$, and thus we get \eqref{eqn: bracketErgodicity}. By the ergodicity property of Lemma \ref{Lem U ergodicity}, we also have 
	$$\frac{1}{T} \int_{[0,uT]}    \lambda^\infty(t)^{-1}U^\infty(t;\theta^*)^2  dt \to^P \E \left[\lambda^\infty(0)^{-1}U^\infty(0;\theta^*)^2 \right] = uc^\T \Omega c,$$
	where $\Omega = \E[\lambda^\infty(0)^{-1} \partial_\psi \lambda^\infty(0)\partial_\psi \lambda^\infty(0)^\T ]$, which proves our claim.

	Next, for Lindeberg's condition, for any $a>0$, similarly to \cite{ClinetYoshida2017}
	\begin{eqnarray*}
		\E[\sum_{s \le u}(\Delta S_s^T)^2\mathbf{1}_{\{\Delta S_s^T>a\}}]  &\le& \frac{1}{a^2}\E[\sum_{s \le u}(\Delta S_s^T)^{4}]\\
		&=& \frac{1}{a^2}\E[\int_{[0,uT]}\lvert \frac{1}{\sqrt{T}}\lambda(t;\theta^{*})^{-1} U(t;\theta^*, \phi^*)\rvert^{4} N(dt)]\\
		&=& \frac{1}{a^2T^2}\E[\int_{[0,uT]} \lambda(t;\theta^{*})^{-3} \lvert U(t;\theta^*, \phi^*)\rvert^{4}dt]\\
		&\leq& \frac{uK}{T}   \sup_{t \in \mathbb{R}_+} \E  \lvert U(t;\theta^*, \phi^*)\rvert^{4} \to0,\\
	\end{eqnarray*}
	where we have used Condition 3 along with the boundedness of $\lambda(t;\theta^{*})^{-1}$. As in \cite{ClinetYoshida2017}, application of \citep[3.24 chapter VIII]{jacod2013limit} gives the required functional CLT.

\end{proof}
\begin{lem} \label{Lem 4}														
	\begin{equation*}
	\frac{1}{\sqrt{T}}(\p_\psi l_g(\hat \nu_T)-\p_\psi l_g(\nu^{*})) \to^P 0.
	\end{equation*}
\end{lem}
\begin{proof}
	Rewrite
	\begin{multline}\label{Eqn: Score Est - Score True}
	\frac{1}{\sqrt{T}} c^\T(\p_\psi l_g(\hat \nu_T)-\p_\psi l_g(\nu^{*})) \\
	=\frac{1}{\sqrt{T}} c^\T(\p_\psi l_g(\hat \theta_T, \hat \phi_T,0)-\p_\psi l_g(\hat\theta_T,\phi^*,0))
	+\frac{1}{\sqrt{T}} c^\T (\p_\psi l_g(\hat \theta_T, \phi^*,0)-\p_\psi l_g(\theta^*, \phi^*,0))
	\end{multline}
	Consider the first term in \eqref{Eqn: Score Est - Score True}. Recall that $\hat G_c(\X)-G_c(\X) = \hat \mu_H-\mu_H$, we have  
	\begin{align*}
	&U(t;\hat \theta_T, \hat \phi_T)-U(t;\hat \theta_T, \phi^*)\\
	&=c^\T\p_\psi \lambda_g(t;\hat \theta_T, \hat \phi_T,0)-c^\T\p_\psi \lambda_g(t;\hat\theta_T, \phi^*,0) \\
	&= \hat \vartheta_T \int_{[0,t)\times \mathbb{X}}w(t-s;\hat \alpha_T)[\hat G_c(\x)-G_c(\x)]N_g^0(ds \times d\x) \\
	& = \hat \vartheta_T \int_{[0,t)}w(t-s;\hat \alpha_T)N(ds) c^\T(\hat \mu_H-\mu_H)
	\end{align*}
	giving
	\begin{align*}
	&\frac{1}{\sqrt{T}} c^\T(\p_\psi l_g(\hat \theta_T, \hat \phi_T)-\p_\psi l_g(\hat\theta_T,\phi^*))\\
	&= 	\frac{1}{\sqrt{T}} \hat \vartheta_T \int_{[0,T]}\lambda(t;\hat\theta_T)^{-1}\int_{[0,t)}  w(t-s;\hat\alpha_T) N(ds) \{N(dt)- \lambda(t;\hat \theta_T) dt\} c^\T(\hat \mu_H-\mu_H)
	\end{align*}	
	Now by Condition \ref{Cond: Marks Stat Ergodic}, $\hat \mu_H -\mu_H\to^P 0$. Also, using the consistency of the quasi likelihood estimates for the unmarked process, $\hat \vartheta \to^P \vartheta^*$. Finally
	\[
	\frac{1}{\sqrt{T}}\int_{[0,T]}\lambda(t;\hat\theta_T)^{-1}\int_{[0,t)} w(t-s;\hat\alpha_T)N(ds)  \{N(dt)- \lambda(t;\hat \theta_T) dt\}
	\]
	is precisely the same as the derivative of the nonboosted likelihood w.r.t. the branching ratio parameter $\vartheta$ and it converges in distribution to a normal random variable directly from \citep[Proof of Theorem 3.11]{ClinetYoshida2017}. Hence the first term in \eqref{Eqn: Score Est - Score True} converges to zero in probability.
	
	Consider the second term in \eqref{Eqn: Score Est - Score True} which is written as
	\begin{align*}
	&\frac{1}{\sqrt{T}} c^\T [\p_\psi l_g(\hat \theta_T, \phi^*,0)-\p_\psi l_g(\theta^*, \phi^*,0)]\\
	&=\frac{1}{\sqrt{T}}\left\lbrace \int_{[0,T]}\lambda(t;\hat \theta_T)^{-1} U(t;\hat \theta_T)N(dt)-\int_{[0,T]}U(t;\hat \theta_T)dt\right\rbrace\\&\quad-\frac{1}{\sqrt{T}}\left\lbrace\int_{[0,T]}\lambda(t; \theta^*)^{-1} U(t; \theta^*)N(dt)-\int_{[0,T]}U(t;\theta^*)dt\right\rbrace\\
	&= \frac{1}{T}\left\lbrace \int_{[0,T]}\p_\theta(\lambda(t;\bar \theta_T)^{-1} U(t;\bar\theta_T))N(dt)-\int_{[0,T]}\p_\theta U(t;\bar \theta_T)dt\right\rbrace \sqrt{T}(\hat \theta_T - \theta^*)\\
	\end{align*}
	using a first order Taylor series expansion where $\bar \theta_T \in [\theta^*, \hat \theta_T]$. By the central limit theorem in \cite{ClinetYoshida2017} $\sqrt{T}(\hat \theta_T - \theta^*)$ is asymptotically normal. We show that the term multiplying this converges to zero in probability using a similar argument as to that in \citep[Proof of Lemma 3.12]{ClinetYoshida2017}.  Now, at any $\theta$ we have
	\begin{align*}
	&\frac{1}{T}\left\lbrace \int_{[0,T]}\p_\theta\{\lambda(t; \theta)^{-1} U(t;\theta))\}N(dt)-\int_{[0,T]}\p_\theta U(t; \theta)dt\right\rbrace \nonumber\\
	& = \frac{1}{T}\int_{[0,T]}\p_{\theta}\{\lambda(t; \theta)^{-1} U(t;\theta)\} \tilde N(dt)\\
	&-\frac{1}{T}\int_{[0,T]}\lambda(t,\theta)^{-2}\p_\theta\lambda(t;\theta) U(t;\theta)\lambda(t;\theta^*)dt\\
	&-\frac{1}{T}\int_{[0,T]}\p_\theta U(t;\theta)\lambda(t;\theta)^{-1}[\lambda(t;\theta)-\lambda(t;\theta^*)]dt
	\end{align*}
	These three terms are analogous to the three terms appearing in the expression for $\p_\theta^2 l_T(\theta)$ in \citep[middle p. 1809]{ClinetYoshida2017} and are listed in the same order.
	
	The third term converges in probability to zero uniformly on a ball, $V_T$ centred on $\theta^*$ shrinking to $\{\theta^*\}$ using similar arguments to those in \citep[p. 1810]{ClinetYoshida2017} for their third term and Lemma \ref{Lem U ergodicity}.
	
	The second term also converges to a limit uniformly on a ball, $V_T$ centred on $\theta^*$ shrinking to $\{\theta^*\}$ and uses ergodicity from Lemma \ref{Lem U ergodicity} and similar arguments to \cite{ClinetYoshida2017} but note that the limit is a matrix of zeros because its expectation is zero corresponding to the block diagonal structure of the full information matrix.
	
	Finally consider the first, martingale term,
	\begin{equation*}
	M_T(\theta)= \frac{1}{T}\int_{[0,T]}\p_{\theta}\{\lambda(t; \theta)^{-1} U(t;\theta)\} \tilde N(dt)
	\end{equation*}
	which we will show converges to zero in probability uniformly in $\theta \in \Theta$ (uniformity allow us to deal with the evaluation at $\bar \theta_T$) and use $\E[\left| M_{a,T}(\bar\theta_T)\right|^p] \le \E[\sup_{\theta \in \Theta }\left|M_{a,T}(\theta)\right|^p]$ where $M_{a,T}$ is the $a$'th component. 
	For $p = dim(\Theta)+1$
	\begin{equation*}
	\E[\sup_{\theta \in \Theta }\left|M_{a,T}(\theta)\right|^p]
	\le K(\Theta,p) \left\lbrace \int_{\Theta}d\theta \E[|M_{T}(\theta)|^p] + \int_{\Theta}d\theta \E[|\p_\theta M_{T}(\theta)|^p]\right\rbrace 
	\end{equation*} 
	where $K(\Theta,p) <\infty$ using Sobolev's inequality as in \citep[Proof of Lemma 3.10]{ClinetYoshida2017}. We next apply the Davis-Burkholder-Gundy inequality followed by Jensen's inequality to each of $\E[|M_T(\theta)|^p]$ and $\E[|\p_\theta M_T(\theta)|^p]$.
	
	First
	\begin{align*}
	\E[|M_T(\theta)|^p]&\le C T^{-p}\E\left[ \int_{[0,T]}(\p_{\theta}\{\lambda(t; \theta)^{-1} U(t;\theta)\})^2  \lambda(t;\theta^*)dt\right]^\frac{p}{2}\\
	& \le CT^{-p+\frac{p}{2}-1}   \int_{[0,T]}\E\left[|\p_{\theta}\{\lambda(t; \theta)^{-1} U(t;\theta)\}|^p  \lambda(t;\theta^*)^\frac{p}{2}\right]dt\\
	& \le C T^{-\frac{p}{2}} \sup_{t \in \Rp}\E\left[ \sup_{\theta \in \Theta} |\p_{\theta}\{\lambda(t; \theta)^{-1} U(t;\theta)\}|
	^p  \lambda(t;\theta^*)^\frac{p}{2}  \right]
	\end{align*}
	Similarly
	\begin{align*}
	\E[|\p_{\theta} M_T(\theta)|^p]&\le C T^{-p}\E\left[ \int_{[0,T]}(\p_{\theta}^2\{\lambda(t; \theta)^{-1} U(t;\theta)\})^2  \lambda(t;\theta^*)dt\right]^\frac{p}{2}\\
	& \le C T^{-\frac{p}{2}} \sup_{t \in \Rp}\E\left[ \sup_{\theta \in \Theta} |\p_{\theta}^2\{\lambda(t; \theta)^{-1} U(t;\theta)\}|
	^p  \lambda(t;\theta^*)^\frac{p}{2}  \right]
	\end{align*}
	Now, as in \cite{ClinetYoshida2017} proof of Lemma 3.12, the processes $ |\p_{\theta}\{\lambda(t; \theta)^{-1} U(t;\theta)\}|
	^p  \lambda(t;\theta^*)^\frac{p}{2} $ and $|\p_{\theta}^2\{\lambda(t; \theta)^{-1} U(t;\theta)\}|
	^p  \lambda(t;\theta^*)^\frac{p}{2}$
	are dominated by polynomials in $\lambda(t;\theta)^{-1}$, $\p_\theta^i \lambda(t;\theta)$ and $\p_\theta^i U(t;\theta)$ for $i\in{0,1,2}$. The first two terms are covered by \cite{ClinetYoshida2017} condition A2 and shown to be true by them for the exponential decay Hawkes process. The terms $\p_\theta^i U(t;\theta)$ are covered by Condition \ref{Cond: Condition on sup lambda_g} which is shown to be true for the exponential decay model in Lemma \ref{Lem 1}.

\end{proof}
\begin{lem} \label{Lem 5}
	The estimated information matrix, $\hat \I_{\psi}$  defined in \eqref{Eqn: Estimated Ipsipsi under H0, ver 3} satisfies
	\begin{equation*}\label{Eqn: Asympotitic Inf Mat Gamma}
	\frac{1}{T}\hat \I_{\psi}\to^P \Omega.
	\end{equation*}
\end{lem}
\begin{proof}
Recall from \eqref{Eqn: Estimated Ipsipsi under H0, ver 3}
\[
\hat \I_{\psi} = \int_{[0,T]}  \lambda(t;
\hat\theta_T)^{-2}(\p_\psi\lambda_g(t;\hat\nu_T))^{\otimes 2} N(dt)
\]
and let
\[	
\hat \I_{\psi}(\nu^*) =\int_{[0,T]}  \lambda(t;\theta^{*})^{-2}(\p_\psi\lambda_g(t;\nu^{*}))^{\otimes 2} N(dt).	
\]
Note that, by similar arguments to that of the proof of Lemma 3.12 in \cite{ClinetYoshida2017}, we have 
\[
T^{-1}\hat \I_{\psi}(\nu^*) = T^{-1}\int_{[0,T]}  \lambda(t;\theta^{*})^{-1}(\p_\psi\lambda_g(t;\nu^{*}))^{\otimes 2}dt + M_T,
\]
where $M_T$ is a martingale of order $O_P(T^{-1/2})$. by ergodicity, we thus have that $T^{-1}\hat \I_{\psi}(\nu^*)\to \Omega$ where $\Omega$ is the same positive definite matrix as in Lemma \ref{Lem 3}. Hence to prove Lemma \ref{Lem 5} it is sufficient to show that $\frac{1}{T}c^\T\{\hat \I_{\psi}-\I_{\psi}(\nu^*)\}c\to 0$ for any $c \in \mathbb{R}^r$. 
Let $R(t;\theta,\phi)  = \lambda(t;\theta)^{-1}U(t;\theta,\phi)$. Then
\begin{align*}
 \frac{1}{T}c^\T\{\hat \I_{\psi}-\hat \I_{\psi}(\nu^*)\}c &=\frac{1}{T} \int_{[0,T]} \{R(t;\hat \theta_T, \hat \phi_T)^2 - R(t;\theta^*, \phi^*)^2\}N(dt)\\
 &=\frac{1}{T} \int_{[0,T]} \{R(t;\hat \theta_T, \hat \phi_T)^2 - R(t;\hat \theta_T, \phi^*)^2\}N(dt)\\
 & \quad + \frac{1}{T} \int_{[0,T]} \{R(t;\hat \theta_T, \phi^*)^2 - R(t; \theta^*, \phi^*)^2\}N(dt)
\end{align*}
Now, using a Taylor series expansion
\begin{align*}
&\frac{1}{T} \int_{[0,T]} \{R(t;\hat \theta_T, \hat \phi_T)^2 - R(t;\hat \theta_T, \phi^*)^2\}N(dt)\\
&\quad =2c^\T(\hat \mu_H - \mu_H)\frac{1}{T} \int_{[0,T]}\{\lambda(t;\hat \theta_T)^{-1}\int_{[0,t)}\hat\vartheta_T w(t-s;\hat \alpha_T)N(ds)\} R(t;\hat \theta_T, \phi^*)N(dt) \\
&\quad \quad+\{c^\T(\hat \mu_H - \mu_H)\}^2\frac{1}{T} \int_{[0,T]}\{\lambda(t;\hat \theta_T)^{-1}\int_{[0,t)}\hat\vartheta_T w(t-s;\hat \alpha_T)N(ds)\}^2 N(dt).
\end{align*}
Now $c^\T(\hat \mu_H - \mu_H)\to^P 0$ and, similarly to \cite{ClinetYoshida2017}, both integrals are uniformly bounded in probability for all $T$ hence $\frac{1}{T}c^\T\{\hat \I_{\psi}-\I_{\psi}(\nu^*)\}c$ converges to zero in probability, completing the proof.
\end{proof}
\section{Proof of Theorem \ref{thm LocalPower}}\label{Sec: Proof Thm 2}
We have divided the proof of Theorem \ref{thm LocalPower} into a series of Lemmas. Before we derive the asymptotic distribution of the score statistic we need some definitions. For the sake of simplicity, we will use the notation $\lambda^{T}(.;\theta) := \lambda_g^{T}(.;\theta,\phi,0)$ (which is independent of $\phi \in \Phi$). By Proposition \ref{propExistence}, we may assume the existence of an unmarked Hawkes process $N^{(0)}$ generated by the same measure $\bar{N}_g$ on $\mathbb{R}^2 \times \mathbb{X}$ as the sequence of processes $N_g^T$. $N^{(0)}$ is thus a marked Hawkes process with mark function $g(., \phi^*, 0) = 1$. We call $\lambda^{(0)}$ its associated stochastic intensity, that is for any $\theta \in \Theta$ 
\begin{eqnarray*}
	\lambda^{(0)}(t;\theta) &=& \eta + \vartheta \int_{(-\infty,t) \times \mathbb{X} } w(t-s;\alpha) g(\x; \phi^*,0)N^{(0)}(ds \times d\x)\\
	&=& \eta + \vartheta \int_{(-\infty,t) } w(t-s;\alpha) N^{(0)}(ds \times \mathbb{X})\\
\end{eqnarray*}
where $\lambda^{(0)}(t;\theta^*)$ is the actual stochastic intensity of $N^{(0)}$, that is $\int_0^t \lambda^{(0)}(s;\theta^*) ds$ is the predictable compensator of $N^{(0)}([0,t] \times \mathbb{X})$. Finally, we define for $i \in \{0,1\}$, $\theta \in \Theta$, $\phi \in \Phi$,
$$\lambda^{(0),i}(t;\theta, \phi) = \vartheta \int_{[0,t)\times\mathbb{X}} w(t-s;\alpha) \partial_\psi^i g(\x;\phi,0)N^{(0)}(ds \times d\x).$$

We first show that in the sense of (\ref{deviation local alternative 1}) and (\ref{deviation local alternative 3}) below, $N_g^T$ converges to $N^{(0)}$ when $T \to +\infty$.

\begin{lem} \label{lem deviation alternatives} Let $f$ be a predictable process depending on $\theta \in \Theta$ such that $$\sup_{t \in [0,T]} \E \sup_{\theta \in \Theta}|f(t,\theta)|^p < +\infty$$ for some $p \geq 2$. Then we have
	\begin{eqnarray} \label{deviation local alternative 1}
	\E \sup_{\theta \in \theta} \left|\int_{[0,T) \times \mathbb{X}} f(t,\theta) \{N_g^T - N^{(0)}\}(dt \times d\x)\right| = O\left(T^{1/2}\right)
	\end{eqnarray} 
	and for any $i \in \{0,1\}$
	\begin{eqnarray}\label{deviation local alternative 3}
	\sup_{t \in \mathbb{R}_+} \E \sup_{\theta \in \theta} | \partial_\psi^i \lambda_g^{T}(t;\theta,\phi^*,0) - \lambda^{(0),i}(t;\theta,\phi^*) |^2 = O\left(T^{-1/2}\right).
	\end{eqnarray}
\end{lem}

\begin{proof}
	We prove our claim in three steps.\\
	\textbf{Step 1.} Letting $\delta^T(t) = \E|\lambda_g^{T}(t;\theta^*,\phi^*,\psi_T^*) - \lambda^{(0)}(t;\theta^*)|$, we prove $\sup_{t \in [0,T]}\delta^T(t)  = O(T^{-1/2})$. We have 
	\begin{eqnarray*}
		\delta^T(t) &\leq&\vartheta\E  \int_{[0,t) \times \mathbb{X}} w(t-s; \alpha^*) g(\x;\phi^*,\psi_T^*) |N_g^T - N^{(0)}|(ds \times d\x)\\
		&+& \vartheta\E  \int_{[0,t) \times \mathbb{X}} w(t-s; \alpha^*) |g(\x;\phi^*,\psi_T^*) - 1| N^{(0)}(ds \times d\x)\\
		&\leq&  \vartheta \E  \int_{[0,t)  } w(t-s; \alpha^*)  \E[g(\y_s;\phi^*,\psi_T^*) | \calf_{s-}^{\y}] |\lambda_g^{T}(s;\theta^*,\phi^*,\psi_T^*) - \lambda^{(0)}(s;\theta^*)|  ds\\
		&+& T^{-1/2}|\gamma^*|\vartheta\E  \int_{[0,t) \times \mathbb{X}} w(t-s; \alpha^*) |\sup_{\psi \in [0, \psi_T^*]} \partial_\psi g(\x;\phi^*,\psi)| N^{(0)}(ds \times d\x)\\
		&\leq&  \vartheta \E  \int_{[0,t)  } w(t-s; \alpha^*)  \E[g(\y_s;\phi^*,\psi_T^*) | \calf_{s-}^{\y}] |\lambda_g^{T}(s;\theta^*,\phi^*,\psi_T^*) - \lambda^{(0)}(s;\theta^*)|  ds\\
		&+& T^{-1/2}K\\
		&\leq& C \vartheta  \int_{[0,t) } w(t-s; \alpha^*) \delta^T(s)ds + T^{-1/2}K\\
		&\leq & C \vartheta \sup_{s \in [0,T]} \delta^T(s) + T^{-1/2}K,
	\end{eqnarray*}   
	for some constant $K >0$, where we have used that $\E[g(\y_s;\phi^*,\psi_T^*) | \calf_{s-}^{\y}] \leq C < 1/\vartheta$, that $\int_0^{+\infty}w(.;\alpha^*) = 1$, and Condition \ref{conditionLocalPower}. Moreover, for a vector $x$, we have used the notation $|x| = \sum_i |x_i|$. Taking the supremum over $[0,T]$ on the left hand side, we deduce $\sup_{s \in [0,T]} \delta^T(s) \leq KT^{-1/2}/(1-C\vartheta)$ and we are done.\\
	\textbf{Step 2.} Letting $\epsilon^T(t) = \E|\lambda_g^{T}(t;\theta^*,\phi^*,\psi_T^*) - \lambda^{(0)}(t;\theta^*)|^2$, we prove $\sup_{t \in [0,T]}\epsilon^T(t)  = O(T^{-1/2})$. We have for some $c > 0$ arbitrary small,
	\begin{eqnarray*}
		\epsilon^T(t) &\leq& (1+c)\vartheta^2\E  \left|\int_{[0,t) \times \mathbb{X}} w(t-s; \alpha^*) g(\x;\phi^*,\psi_T^*) (N_g^T - N^{(0)})(ds \times d\x)\right|^2 \\
		&+& (1+c^{-1})\vartheta^2 \E \left| \int_{[0,t) \times \mathbb{X}} w(t-s; \alpha^*) |g(\x;\phi^*,\psi_T^*) - 1| N^{(0)}(ds \times d\x)\right|^2,\\
		&=& I + II, 
	\end{eqnarray*}
	where we have used the inequality $(x+y)^2 \leq (1+c)x^2+(1+c^{-1})y^2$ for any $c > 0$. First, we have 
	\begin{eqnarray*}
		I &\leq & (1+c)(1+c^{-1})\vartheta^2 \E  \left|\int_{[0,t) \times \mathbb{X}} w(t-s; \alpha^*) g(\x;\phi^*,\psi_T^*) (\tilde{N}_g^{T} - \tilde{N}^{(0)})(ds \times d\x)\right|^2 \\
		&+& (1+c)^2\vartheta^2 \E  \left|\int_{[0,t)} w(t-s; \alpha^*) \E[g(\y_s;\phi^*,\psi_T^*) | \calf_{s-}^\y] (\lambda_g^{T}(s;\theta^*,\phi^*,\psi_T^*) - \lambda^{(0)}(s;\theta^*)) ds \right|^2\\
		&=& I_A + I_B.
	\end{eqnarray*}
	Now, applying Jensen's inequality with respect to the probability measure \newline $w(s;\alpha^*)ds/\int_0^t w(s;\alpha^*)ds$, and then using $\int_0^{+\infty}w(.;\alpha^*) =1$, and \newline $\E[g(\y_{s};\phi^*,\psi_T^*) | \calf_{s-}^\y] \leq C < 1/\vartheta$ yields
	\begin{eqnarray*} 
		I_B &\leq& (1+c)^2\vartheta^2 C^2 \E \int_{[0,t)} w(t-s; \alpha^*)  \epsilon^T(s) ds \\
		& \leq & (1+c)^2 \vartheta^2 C^2 \sup_{s \in [0,T]} \epsilon^T(s).
	\end{eqnarray*}
	Now, for $I_A$, we have
	\begin{eqnarray*}
		I_A &\leq& K \E  \left|\int_{[0,t) \times \mathbb{X}} w(t-s; \alpha^*) (g(\x;\phi^*,\psi_T^*)-1) (\tilde{N}_g^{T} - \tilde{N}^{(0)})(ds \times d\x)\right|^2  \\
		&+& K \E  \left|\int_{[0,t) \times \mathbb{X}} w(t-s; \alpha^*) (\tilde{N}_g^{T} - \tilde{N}^{(0)})(ds \times d\x)\right|^2  \\
		&\leq& K\E \int_{[0,t)} w(t-s; \alpha^*)^2 \E[(g(\y_{s};\phi^*,\psi_T^*) -1)^2 | \calf_{s-}^\y]\\ &&\quad|\lambda_g^{T}(s;\theta^*,\phi^*,\psi_T^*)- \lambda^{(0)}(s;\theta^*) |ds\\
		& &\quad +K \E \int_{[0,t)} w(t-s; \alpha^*)^2 |\lambda_g^{T}(s;\theta^*,\phi^*,\psi_T^*) - \lambda^{(0)}(s;\theta^*) |ds\\
		&\leq&  \frac{K}{T} \int_{[0,t)} w(t-s; \alpha^*)^2 \E[\sup_{\psi \in [0,\psi_T^*]}\partial_\psi g(\y_{s};\phi^*,\psi)^4 ]^{1/2} \epsilon^T(s)^{1/2}ds \\&& \quad  +K \sup_{s \in [0,T]} \delta^T(s)\\
		& \leq & \frac{K}{T} \sup_{s \in [0,T]} \left(\epsilon^T(s) \vee 1\right) + K  \sup_{s \in [0,T]} \delta^T(s),
	\end{eqnarray*}
	where we have used Cauchy-Schwarz inequality along with (\ref{momentLocalPower}) and (\ref{uniform w localpower}). Moreover, following a similar path as for Step 1, we also have that $II \leq K T^{-1}$ by (\ref{momentLocalPower}). Thus, overall, using that $\sup_{s \in [0,T] }\delta^T(s) \leq KT^{-1/2}$ by Step 1, we obtain for some constant $K>0$
	$$ \epsilon^T(t) \leq K(T^{-1} + T^{-1/2}) + ((1+c)^2 \vartheta^2 C^2 + KT^{-1}) \sup_{s \in [0,T]} \epsilon^T(s),$$
	and taking the supremum over $[0,T]$ on the left hand side, taking $c>0$ close enough to $0$ and $T$ large enough so that $((1+c)^2 \vartheta^2 C^2 + KT^{-1}) < A$ for some constant $A <1$, we get 
	$$ \sup_{s \in [0,T]} \epsilon^T(s) \leq  K(T^{-1} + T^{-1/2})/(1-A) \leq \tilde{K} T^{-1/2}$$
	for some $\tilde{K} >0$.
	
	\textbf{Step 3.} We prove (\ref{deviation local alternative 1}) and  (\ref{deviation local alternative 3}). For  (\ref{deviation local alternative 1}), this is a direct consequence of the fact that the compensator of $|N_g^T - N^{(0)}|$ is $\int_0^T |   \lambda_g^{T}(t;\theta^*,\phi^*,\psi_T^*) - \lambda^{(0)}(t;\theta^*) |dt$, Cauchy-Schwarz inequality and the uniform condition on $f$. For (\ref{deviation local alternative 3}), let $i \in \{0,1\}$. We have 
	\begin{align*}
		&\E \sup_{\theta \in \Theta} |\partial_\psi^i \lambda_g^{T}(t;\theta,\phi^*,0) - \lambda^{(0),i}(t;\theta,\phi^*)|^2 \\&\leq K \E  \sup_{\alpha \in \mathcal{A}} \left| \int_{[0,t) \times \mathbb{X}} w(t-s; \alpha) \partial_\psi^i g(\x;\phi^*,0) (N_g^T - N^{(0)})(ds \times d\x) \right|^2 \\
		&\leq  K \E   \left| \int_{[0,t) \times \mathbb{X}} \bar{w}(t-s)|\partial_\psi^i g(\x;\phi^*,0)||N_g^T - N^{(0)}|(ds \times d\x) \right|^2\\
\end{align*}
	And from here, using Burkholder-Davis-Gundy inequality, Step 2 of this proof along with conditions (\ref{momentLocalPower}) and (\ref{uniform w localpower}) we deduce that the above term is dominated by $K T^{-1/2}$ for some $K >0$ uniformly in $t \in \mathbb{R}_+$.
\end{proof} 

\begin{lem} (Consistency of $\hat{\nu}_T$ under the local alternatives) \label{lem consistencyMLE local power}
	Under $H_1^T$, we have
	$$\hat{\nu_T} \to^P \nu^* := (\theta^*, \phi^*, 0).$$ 
\end{lem}
\begin{proof}
	The convergence of the third component is obvious, and the convergence of the second one is assumed. All we have to show is that $\hat{\theta}_T \to^P \theta^*$. Let $l_T^{(0)}(\theta) = \int_{[0,T) \times \mathbb{X}} \log \lambda^0(t;\theta) N^{(0)}(dt \times d\x) - \int_0^T \lambda^{(0)}(t;\theta) dt$, where we recall that $\lambda^{(0)}$ is the stochastic intensity of $N^{(0)}$. We need to show that uniformly in $\theta \in \Theta$, we have the convergence $T^{-1} (l_T(\theta) - l_T^{(0)}(\theta)) \to^P 0$. But note that $T^{-1} (l_T(\theta) - l_T^{(0)}(\theta)) = I + II$ with 
	$$ I = -T^{-1} \int_{[0,T) \times \mathbb{X}} \left\{\log \lambda^{(0)}(t;\theta) N^{(0)}(dt \times d\x) - \log \lambda^{T}(t;\theta) N_g^T(dt \times d\x) \right\} $$
	and 
	$$ II =  -T^{-1}\int_0^T \{\lambda^{(0)}(t;\theta) - \lambda^{T}(t;\theta) \}dt. $$
	By (\ref{deviation local alternative 3}), we immediately have that $\E \sup_{\theta \in \Theta} |\{\lambda^{(0)}(t;\theta) - \lambda^{T}(t;\theta)| = O(T^{-1/2})$ uniformly in $t \in [0,T]$, so that $II \to^P 0 $ uniformly in $\theta \in \Theta$. Writing $I$ as the sum 
	\begin{eqnarray*}
		&& T^{-1} \int_{[0,T) \times \mathbb{X}} \left\{\log \lambda^{(0)}(t;\theta)  - \log \lambda^{T}(t;\theta) \right\}N^{(0)}(dt \times d\x) \\
		&+& T^{-1} \int_{[0,T) \times \mathbb{X}} \log \lambda^{T}(t;\theta) \left\{N^{(0)}(dt \times d\x) - N_g^T(dt \times d\x) \right\}\\
		&=&  A + B,
	\end{eqnarray*}
	we need to show that both terms tend to $0$. Since $|\log \lambda^{(0)}(t;\theta)  - \log \lambda^{T}(t;\theta)| \leq \underline{\eta}^{-1} | \lambda^{(0)}(t;\theta)  - \lambda^{T}(t;\theta)|$, we easily get by Cauchy-Schwarz inequality and  (\ref{deviation local alternative 3}) that $\E \sup_{\theta \in \Theta} |A| \to 0$. Moreover, using $\log \lambda^{T}(t;\theta) \leq \lambda^{T}(t;\theta) - 1$, by (\ref{deviation local alternative 1}) and Condition \ref{conditionLocalPower} we have that $\E \sup_{\theta \in \Theta} |B| \to 0$ and we are done. 
	
\end{proof}

\begin{lem} \label{lem local power}
	Under $H_1^T$, we have 
	$$T^{-1/2}\partial_\psi l_g(\hat{\nu}_T) \to^d \mathcal{N}(\Omega \gamma^*, \Omega)$$ 
\end{lem}

\begin{proof}
	First, note that by application of Lemma \ref{lem deviation alternatives}, Lemma \ref{lem consistencyMLE local power}, and following the same path as for the proof of Lemma \ref{Lem 4}, we deduce 
	$$T^{-1/2}\partial_\psi l_g(\hat{\nu}_T) - T^{-1/2}\partial_\psi l_g(\nu^*) \to^P 0.$$
	Next, We have 
	\begin{eqnarray*}
		T^{-1/2}\partial_\psi l_g(\theta^*,\phi^*,0) &=& T^{-1/2}\int_{(0,T)} \p_\psi \log \lambda_g^{T}(t;\theta^*,\phi^*,0) \tilde{N}_g^{T}(dt) \\&+& T^{-1/2} \int_{(0,T)} \partial_\psi \lambda_g^{T}(t;\theta^*,\phi^*,0)\left(\frac{\lambda_g^{T}(t;\theta^*,\phi^*,\psi_T^*)}{\lambda^T(t;\theta^*)}-1\right)  dt\\
		&=& I + II,
	\end{eqnarray*}
	where we have used the notation $\tilde{N}_g^{T}(dt) = N_g^T(dt, \mathbb{X}) - \lambda_g^{T}(t; \theta^*,\phi^*, \psi_T^*)dt$. We derive the limit of the first term following the same path as for the proof of Lemma \ref{Lem 3}. Letting $S_u^T = T^{-1/2}\int_{(0,uT)} \p_\psi \log \lambda_g^{T}(t;\theta^*,\phi^*,0) \tilde{N}_g^{T}(dt)$, we directly have that
	$$ \langle S^T,S^T\rangle_u = T^{-1} \int_0^{uT} \frac{\p_\psi \lambda_g^{T}(t;\theta^*,\phi^*,0) \p_\psi \lambda_g^{T}(t;\theta^*,\phi^*,0)^\T}{\lambda^{T}(t;\theta^*)^2} \lambda_g^{T}(t; \theta^*,\phi^*, \psi_T^*) dt.$$
	By (\ref{deviation local alternative 3}), the boundedness of moments of $\lambda_g^{T}$ and its derivatives and H{\"older's inequality} we easily deduce that
	$$ \langle S^T,S^T\rangle_u = T^{-1} \int_0^{uT} \frac{ \lambda^{(0),1}(t;\theta^*,\phi^*) \p_\psi \lambda^{(0),1}(t;\theta^*,\phi^*)^\T}{\lambda^{(0)}(t;\theta^*)}dt + o_P(1),$$ 
	which converges in probability to $u\Omega$ by Lemma \ref{Lem 3}. Similarly, Lindeberg's condition $\E \sum_{s \leq u} (\Delta S_s^T)^2 \mathbf{1}_{\{ |\Delta S_s^T| > a\}} \to 0$ for any $a >0$ is satisfied, so that by 3.24, Chapter VIII in \cite{jacod2013limit}, we get that $I = S_1^T \to^d \mathcal{N}(0,\Omega)$. Now we derive the limit for $II$. We have for some $\tilde{\gamma}_T \in [0, \hat{\gamma}_T]$
	
	\begin{eqnarray*}
		II = T^{-1} \int_{(0,T)} \lambda^{T}(t;\theta^*,\phi^*)^{-1} \partial_\psi \lambda_g^{T}(t;\theta^*,\phi^*,0) \partial_\psi \lambda_g^{T}(t;\theta^*,\phi^*,\tilde{\gamma}_T)^\T \gamma^*  dt.
	\end{eqnarray*}
	Now, using H{\"older's inequality}, the uniform boundedness of moments of $\lambda_g^{T}$ in $\nu$, and (\ref{deviation local alternative 3}), we deduce as previously that 
	$$ II = T^{-1} \int_{(0,T)} \lambda^{(0)}(t;\theta^*)^{-1}   \lambda^{(0),1}(t;\theta^*,\phi^*)  \lambda^{(0),1}(t;\theta^*,\phi^*)^\T \gamma^* dt + o_P(1),$$
	which, by the proof of Lemma \ref{Lem 5}, tends in probability to the limit $\Omega \gamma^*$. By Slutsky's Lemma, we get the desired convergence in distribution for $T^{-1/2}\partial_\psi l_g(\hat{\nu}_T)$. 
\end{proof}

\begin{lem} \label{lem Fisher local power} 
	Under $H_1^T$, we have 
	$$ T^{-1} \hat{\I}_\psi \to^P \Omega.$$
\end{lem}

\begin{proof}
	First, as for Lemma \ref{lem local power}, note that by application of Lemma \ref{lem deviation alternatives}, Lemma \ref{lem consistencyMLE local power}, and following the same path as for the proof of Lemma \ref{Lem 5}, we have 
	$$T^{-1}( \hat{\I}_\psi - \hat{\I}_\psi(\nu^*)) \to^P 0.$$
	Now recall that 
	$$T^{-1}\hat{\I}_\psi(\nu^*)  = T^{-1}\int_{[0,T] \times \mathbb{X}}  \lambda^T(t;\theta^{*})^{-2}(\p_\psi\lambda_g^{T}(t;\nu^{*}))^{\otimes 2} N_g^T(dt \times d\x).$$
	By (\ref{deviation local alternative 1}), (\ref{deviation local alternative 3}), the boundedness of moments of $\lambda_g^{T}$ and its derivatives and H{\"older's inequality} we get 
	$$T^{-1}\hat{\I}_\psi(\nu^*) = T^{-1}\int_{[0,T] \times \mathbb{X}}  \lambda^{(0)}(t;\theta^{*})^{-2}(\lambda^{(0),1}(t;\nu^{*}))^{\otimes 2} N^{(0)}(dt \times d\x) + o_P(1),$$
	and by Lemma \ref{Lem 5}, the right-hand side converges in probability to $\Omega$.
\end{proof} 
\section*{Acknowledgements}
K-A Richards gratefully acknowledges PhD scholarship support by Boronia Capital Pty. Ltd., Sydney, Australia. The research of S. Clinet is supported by a special grant from Keio University. W. T.M. Dunsmuir was supported by travel funds from the Faculty of Sciences, University of New South Wales.
\bibliography{ScoreTestMarksinHawkesSEPP}
\bibliographystyle{chicago}
\end{document}